\documentclass[12pt,reqno]{amsart}

\usepackage[usenames]{color}
\usepackage{amssymb}
\usepackage{graphicx}
\usepackage{amscd}
\usepackage{url}
\usepackage[colorlinks=true,
linkcolor=webgreen,
filecolor=webbrown,
citecolor=webgreen]{hyperref}

\definecolor{webgreen}{rgb}{0,.5,0}
\definecolor{webbrown}{rgb}{.6,0,0}

\usepackage{color}
\usepackage{fullpage}
\usepackage{float}
\usepackage{comment}
\usepackage{url}
\usepackage{graphics,amsmath,amssymb}
\newcommand{\seqnum}[1]{\href{http://oeis.org/#1}{\underline{#1}}}
\usepackage{amsthm}
\usepackage{amsfonts}
\usepackage{latexsym}
\usepackage{enumerate}

\theoremstyle{plain}
\newtheorem{theorem}{Theorem}

\newtheorem{lemma}[theorem]{Lemma}
\newtheorem{proposition}[theorem]{Proposition}

\theoremstyle{remark}

\allowdisplaybreaks
\newcommand{\ress}[2]{{\rm{Res}}\big(#1,#2\big)}
\newcommand{\dis}[1]{{\rm{Disc}}\big(#1\big)}
\newcommand{\Ft}[1]{\mathcal{F}_{#1}}
\newcommand{\Lt}[1]{\mathcal{L}_{#1}}
\newcommand{\lc}{{\rm lc}}

\begin{document}
	
\title[The resultant, the discriminant, and the derivative of generalized Fibonacci polynomials]
       {The resultant, the discriminant, and the derivative of generalized Fibonacci polynomials}

\author[Rigoberto Fl\'orez]{Rigoberto Fl\'orez }
\address{Department of Mathematical Sciences\\ The Citadel\\ Charleston, SC \\U.S.A}
\email{rigo.florez@citadel.edu}
\thanks{}

\author[Robinson Higuita]{Robinson Higuita }
\address{Instituto de Matem\'aticas\\ Universidad de Antioquia\\ Medell\'in\\ Colombia}
\email{robinson.higuita@udea.edu.co}
\thanks{}

\author[Alexander  Ram\'irez]{Alexander  Ram\'irez}
\address{Instituto de Matem\'aticas\\ Universidad de Antioquia\\ Medell\'in\\ Colombia}
\email{jalexander.ramirez@udea.edu.co}

\begin{abstract}
A second order polynomial sequence is of \emph{Fibonacci-type} (\emph{Lucas-type}) if its Binet formula
has a structure similar to that for Fibonacci (Lucas) numbers. Known examples of these type of sequences are:
Fibonacci polynomials, Pell polynomials, Fermat polynomials, Chebyshev polynomials, Morgan-Voyce polynomials, 
Lucas polynomials, Pell-Lucas polynomials, Fermat-Lucas polynomials, Chebyshev polynomials.

The \emph{resultant} of two polynomials is the determinant of the Sylvester matrix and
the \emph{discriminant} of a polynomial $p$ is the resultant of  $p$ and its derivative. We study the resultant, 
 the discriminant, and the derivatives of Fibonacci-type polynomials and Lucas-type polynomials as well combinations of those two types. 
 As a corollary we give explicit formulas for the resultant,  the discriminant, and the derivative for the known polynomials mentioned above. 
\end{abstract}


\maketitle

\section{Introduction}
	
A second order polynomial sequence is of \emph{Fibonacci-type} (\emph{Lucas-type}) if its Binet formula
has a structure similar to that for Fibonacci (Lucas) numbers. Those are known as \emph{generalized Fibonacci polynomials}  GFP
(see for example \cite{Richard, florezHiguitaMuk2018, florezHiguitaMuk:StarD, hoggatt, HoggattLong}).
Some known examples are: Pell polynomials, Fermat polynomials, Chebyshev polynomials,
Morgan-Voyce polynomials, Lucas polynomials, Pell-Lucas polynomials, Fermat-Lucas polynomials, Chebyshev polynomials, Vieta and Vieta-Lucas polynomials.

The \emph{resultant}, $\rm{Res}(.,.)$, of two polynomials is the determinant of the Sylvester matrix (see \eqref{Syl},
\cite{Akritas, Basu,Gelfand:Discriminants, Serge:algebra} or \cite[p. 426]{Sylvester}).
It is very often in mathematics that we ask the question whether or not two polynomials share a root. In particular, if $p$ and $q$ are two GFPs, we ask the question
whether or not $p$ and $q$ have a common root.  Since the resultant of $p$ and $q$ is also the product of $p$ evaluated
at each root of $q$, the resultant of two GFPs can be used to answer the question. The question can be also answered
finding the greatest common divisor, $\gcd$, of $p$ and $q$ (recall that $\gcd(p, q)=1\iff \rm{Res}(p,q)\ne 0)$ .  However, in the procedure of finding the $\gcd(p, q)$ using the  Euclidean algorithm,
the coefficients of the remainders grow very fast. Therefore, the resultant can be used to reduce computation time of
finding the $\gcd$ \cite{Akritas}.

Several authors have been interested in the resultant. The first formula for the resultant of two  cyclotomic polynomials
was given by Apostol  \cite{Apostol}. Some other papers have been dedicated to the study of resultant of
Chebyshev Polynomials \cite{Dilcher, Jacobs, Louboutin,Yamagishi}.  In this paper we deduce simple closed formulas for the
resultants of a big family of GFPs.  We find the resultants for Fibonacci-type polynomials, the resultants for Lucas-type and the resultant
of combinations of those two types. In particular we find the resultant for both kinds of Chebyshev Polynomials.

The \emph{discriminant} is the resultant of a polynomial and its derivative. If $p$ is a GFP, we ask the question
whether or not $p$ has a repeated root. The discriminant helps to  answer this question. In this paper we find simple
closed formulas for the discriminant of both types of GFP. In addition we generalize the derivative given by 
Falc\'on and Plaza \cite{Falcon} for Fibonacci polynomials to GFP.

Note that the resultant has been used to  the solve systems of polynomial equations (it encapsulates the solutions)  \cite{Basu, Kauers, Lewis,  stiller}.
The resultant can also be used in combination with elimination theory to answer other different types of questions about the multivariable polynomials.
However, in this paper we are not are interested in these types of questions.

 \section{Basic definitions}

In this section we summarize some concepts given by the authors in earlier articles. For example, the authors \cite{florezHiguitaMuk2018}
have studied the polynomial sequences given here. Throughout the paper we consider polynomials in $\mathbb{Q}[x]$.  The polynomials
in this subsection are presented in a formal way.    However, for brevity and if there is no ambiguity after this subsection we avoid these formalities. Thus, we present the polynomials without explicit use of ``$x$".

\subsection{Second order polynomial sequences}
In this section we introduce the generalized Fibonacci polynomial sequences. This definition gives rise to some known
polynomial sequences (see for example, Table \ref{familiarfibonacci} or \cite{florezHiguitaMuk2018,florezHiguitaMuk:StarD,HoggattLong, koshy}).

For the remaining part of this section we reproduce the definitions by Fl\'orez et. al.
\cite{florezHiguitaMuk2018, florezHiguitaMuk:StarD} for generalized Fibonacci polynomials.
We now give the two second order polynomial recurrence relations in which we divide the generalized Fibonacci polynomials (GFP).
\begin{equation}\label{Fibonacci;general:FT}
\Ft{0} (x)=0, \; \Ft{1}(x)= 1,\;  \text{and} \;  \Ft{n}(x)= d(x) \Ft{n - 1}(x) + g(x) \Ft{n - 2}(x) \text{ for } n\ge 2
\end{equation}
where $d(x)$, and $g(x)$ are fixed non-zero polynomials in $\mathbb{Q}[x]$.

We say a polynomial recurrence relation is of \emph{Fibonacci-type} if it satisfies the relation given in
\eqref{Fibonacci;general:FT}, and of \emph{ Lucas-type} if:
\begin{equation}\label{Fibonacci;general:LT}
\Lt{0}(x)=p_{0}, \; \Lt{1}(x)= p_{1}(x),\;  \text{and} \;  \Lt{n}(x)= d(x) \Lt{n - 1}(x) + g(x) \Lt{n - 2}(x) \text{ for } n\ge 2
\end{equation}
where $|p_{0}|=1$ or $2$ and $p_{1}(x)$, $d(x)=\alpha p_{1}(x)$, and $g(x)$ are fixed non-zero  polynomials in $\mathbb{Q}[x]$ with
$\alpha$ an integer of the form $2/p_{0}$.

To use similar notation for \eqref{Fibonacci;general:FT} and \eqref{Fibonacci;general:LT} on certain occasions we write
$p_{0}=0$, $p_{1}(x)=1$ to indicate the initial conditions of Fibonacci-type polynomials. Some known examples
of Fibonacci-type polynomials and  Lucas-type polynomials are in Table \ref{familiarfibonacci} or in
\cite{florezHiguitaMuk2018,HoggattLong,Pell, Fermat,  koshy}.

If $G_{n}$ is either $\Ft{n}$ or $\Lt{n}$ for all $n\ge 0$ and $d^2(x)+4g(x)> 0$ then the explicit formula for the recurrence relations in
 \eqref{Fibonacci;general:FT} and \eqref{Fibonacci;general:LT}  is given by
\begin{equation*}
 G_{n}(x) = t_1 a^{n}(x) + t_2 b^{n}(x)
\end{equation*}
where $a(x)$ and $b(x)$ are the solutions of the quadratic equation associated with the second order
recurrence relation $G_{n}(x)$. That is,  $a(x)$ and $b(x)$ are the solutions of $z^2-d(x)z-g(x)=0$.
If $\alpha=2/p_{0}$, then the Binet formula for Fibonacci-type polynomials is stated in  \eqref{bineformulauno} and the Binet formula
for  Lucas-type polynomials is stated in \eqref{bineformulados}
(for details on the construction of the two Binet formulas see \cite{florezHiguitaMuk2018})
\begin{equation}\label{bineformulauno}
\Ft{n}(x) = \dfrac{a^{n}(x)-b^{n}(x)}{a(x)-b(x)}
\end{equation}
and
\begin{equation}\label{bineformulados}
\Lt{n}(x)=\dfrac{a^{n}(x)+b^{n}(x)}{\alpha}.
\end{equation}

Note that for both types of sequence: $$a(x)+b(x)=d(x), \quad a(x)b(x)= -g(x), \quad \text{ and } \quad a(x)-b(x)=\sqrt{d^2(x)+4g(x)}$$
where $d(x)$ and $g(x)$ are the polynomials defined in \eqref{Fibonacci;general:FT} and \eqref{Fibonacci;general:LT}.

A sequence of  Lucas-type (Fibonacci-type) is \emph{equivalent} or \emph{conjugate} to a sequence of Fibonacci-type (Lucas-type),
if their recursive sequences are determined by the same polynomials $d(x)$ and $g(x)$. Notice that two equivalent polynomials
have the same $a(x)$ and $b(x)$ in their Binet representations. In \cite{florezHiguitaMuk2018, florezHiguitaMuk:StarD}
there are examples of some known equivalent polynomials with their Binet formulas.  The  polynomials in Table \ref{familiarfibonacci} are
organized by pairs of equivalent polynomials.  For instance,  Fibonacci and Lucas, Pell and Pell-Lucas, and so on.

We use $\deg(P)$ and $\lc(P)$ to mean the degree of $P$ and the leading  coefficient of a polynomial $P$.
Most of the following conditions were required in the papers that we are citing. Therefore, we requiere here that
$\gcd(d(x), g(x))=1$ and $\deg(g(x))< \deg(d(x))$ for both type of sequences
 and  that conditions in \eqref{extra:condition} also hold for Lucas types polynomials;
 \begin{equation}\label{extra:condition}
 \gcd(p_{0}, p_{1}(x))=1,  \gcd(p_{0}, d(x))=1, \gcd(p_{0}, g(x))=1,  \text{ and  that degree of } \Lt{1} \ge1.
 \end{equation}

Notice that in the definition of Pell-Lucas we have
$Q_{0}(x)=2$ and $Q_{1}(x)=2x$. Thus, the $\gcd(2,2x)=2\ne 1$.
Therefore, Pell-Lucas does not satisfy the extra conditions that we imposed in \eqref{extra:condition}. So,
to resolve this inconsistency we use $Q^{\prime}_{n}(x)=Q_{n}(x)/2$ instead of $Q_{n}(x)$.
	
	\begin{table}[!ht]
		\begin{center}\scalebox{0.8}{
				\begin{tabular}{|l|l|l|l|} \hline
					Polynomial                      & Initial value     & Initial value	& Recursive Formula 						       \\	
					&$G_0(x)=p_0(x)$  	      &$G_1(x)=p_1(x)$	&$G_{n}(x)= d(x) G_{n - 1}(x) + g(x) G_{n - 2}(x)$ 	   \\  \hline   \hline
					Fibonacci             	      & $0$	     &$1$      	&$F_{n}(x) = x F_{n - 1}(x) + F_{n - 2}(x)$	 	       \\
					Lucas 	             	      &$2$	     & $x$ 	 	&$D_n(x)= x D_{n - 1}(x) + D_{n - 2}(x)$                \\ 						
					Pell			    	      &$0$	     & $1$           &$P_n(x)= 2x P_{n - 1}(x) + P_{n - 2}(x)$               \\
					Pell-Lucas 	    	      &$2$	     &$2x$          &$Q_n(x)= 2x Q_{n - 1}(x) + Q_{n - 2}(x)$               \\
					Pell-Lucas-prime 	      &$1$	     &$x$       	&$Q_n^{\prime}(x)= 2x Q_{n - 1}^{\prime}(x) + Q_{n - 2}^{\prime}(x)$               \\
					Fermat  	                       &$0$	     & $1$      	&$\Phi_n(x)= 3x\Phi_{n-1}(x)-2\Phi_{n-2}(x) $           \\
					Fermat-Lucas	              &$2$	     &$3x$  	&$\vartheta_n(x)=3x\vartheta_{n-1}(x)-2\vartheta_{n-2}(x)$\\
					Chebyshev second kind &$0$	     &$1$       	&$U_n(x)= 2x U_{n-1}(x)-U_{n-2}(x)$  	 	              \\
					Chebyshev first kind       &$1$	     &$x$       	&$T_n(x)= 2x T_{n-1}(x)-T_{n-2}(x)$  \\	 	
					Morgan-Voyce	              &$0$	     &$1$      	&$B_n(x)= (x+2) B_{n-1}(x)-B_{n-2}(x) $  	 	         \\
					Morgan-Voyce 	              &$2$	     &$x+2$       &$C_n(x)= (x+2) C_{n-1}(x)-C_{n-2}(x)$  	 	         \\
					Vieta 		              &$0$ 	     &$1$	        &$V_n(x)=x V_{n-1}(x)-V_{n-2}(x)$ 	    \\
					Vieta-Lucas 		      &$2$ 	     &$x$	        &$v_n(x)=x v_{n-1}(x)-v_{n-2}(x)$      \\
					\hline
			\end{tabular}}
		\end{center}
		\caption{Recurrence relation of some GFP.} \label{familiarfibonacci}
	\end{table}
The familiar examples in Table \ref{familiarfibonacci} satisfy that $\deg\left(d \right)>\deg\left(g \right)$. Therefore, for the rest
of the paper we suppose that $\deg\left(d \right)>\deg\left(g \right)$.

 \subsection{The resultant and the discriminant }
 In this section we use the Sylvester  matrix to define the discriminant of two polynomials.
Let $P$ and $Q$ be non-zero polynomials of degree $n$ and $m$ in $\mathbb{Q}[x]$ with
$$P=a_{n}x^n+ a_{n-1}x^{n-1}+ \dots +a_{1}x+a_{0}  \text{ and  } Q=b_{m}x^m+ b_{m-1}x^{m-1}+ \dots +b_{1}x+b_{0}.$$
The square matrix of  size  $n+m$ given in \eqref{Syl} is called the Sylvester matrix  associated to $P$ and $Q$.  
It is denoted by $\rm{Syl}(P,Q)$.  The resultant of  two polynomials is defined using the Sylvester determinant 
(see for example \cite{Akritas,Basu,Jacobs,Sylvester}). Thus, 
the \emph{resultant} of $P$ and $Q$ denoted by $\ress{P}{Q}$ is the determinant of $\rm{Syl}(P,Q)$.

\begin{equation} \label{Syl}
\rm{Syl}(P,Q) =  \left[ {\begin{array}{cccccccc}
   a_{n}   & a_{n-1}&\cdots  &\cdots &a_{0}   &0	         & \cdots  & 0       \\
   0         & a_{n}   &a_{n-1} &\cdots &\cdots  &a_{0}   & \cdots  & 0       \\
   \vdots & \ddots  &\dots    &\dots    &\dots   &\ddots  & \ddots  & \vdots\\
   \vdots & \ddots  &0          &a_{n}&\cdots  &		 & \cdots  & a_{0} \\
   b_{m} & b_{m-1}&\cdots  &\cdots  &b_{0}   &0	         & \cdots  & 0       \\
   0	    & b_{m}   &b_{m-1}&\cdots  &\cdots  &b_{0}   & \cdots  & 0       \\
 \vdots  & \ddots  &\dots     &\dots    &\dots    &\ddots & \ddots  & \vdots\\
   0 	    & \ddots  &0           &b_{m}  &b_{m-1}&\cdots & \cdots  & b_{0}\\
  \end{array} } \right].
\end{equation}

The resultant can be also expressed in terms of the roots of the two polynomials.
If $\{x_i\}_{i=1}^n$ and $\{y_i\}_{i=1}^m$ are the roots of $P$ and $Q$ in $\mathbb{C}$, respectively, then
the resultant of $P $ and $Q$ is given by
	$
	\ress{P}{Q}=a_n ^m b_m ^n \prod_{i=1}^n \prod_{j=1}^m(x_{i}-y_{j})=a_{n}^m \prod_{i=1}^n Q(x_i)=b_{m}^n \prod_{j=1}^m P(y_j).
	$
The \emph{discriminant} of  $P$,  \dis{P}, is defined by $(-1)^{n(n-1)/2}a_{n}^{2n-2}\Pi_{i\ne j} (x_i-x_j)$. 
The discriminant can also be written as a Vandermonde determinant or as a resultant (see \cite{Serge:algebra}). In this paper we use the following expression. 
If $a_{n}=\lc(P)$, $n=\deg(P)$ and $P^{\prime}$ is the derivative of $P$, then
$$\dis{P}=(-1)^{\frac{n(n-1)}{2}}a_{n}^{-1}\ress{P}{P^{\prime}}.$$

We now introduce some notation used throughout the paper. Let
\begin{equation}\label{Degree:Lead:Def}
	\beta=\lc(d), \quad \lambda=\lc(g),  \quad \eta = \deg(d ), \quad \omega =\deg(g ), \quad \text{ and } \quad \rho=\ress{g }{d }.
\end{equation}
	
 \section{Main Results }\label{Main:Results}

We recall that for brevity throughout the paper we present  the polynomials without the ``$x$".
For example, instead of $\Ft{n}(x)$ and $\Lt{n}(x)$ we use $\Ft{n}$ and $\Lt{n}$. This section
presents the main results in this paper. However, their proofs are not given in this section but they will be given
throughout the paper. We give simple expressions for the resultant of two polynomials of Fibonacci-type,
$\ress{\Ft{n}}{\Ft{m}}$, the resultant of two polynomials of Lucas-type, $\ress{\Lt{n}}{\Lt{m}}$, and the
resultant of two equivalent polynomials (Lucas-type and Fibonacci-type),
$\ress{\Lt{n}}{\Ft{m}}$.  As an application of those theorems we give the discriminants of GFPs.
Finally we construct several tables with the resultant and discriminant of
the known polynomials as corollaries of the main results.

We use the notation $E_{2}(n)$ \label{adic:order}  to represent the \emph{integer exponent base two} of a
positive integer $n$  which is defined to be the largest integer $k$ such that $2^{k}\mid n$
(this concept is also known as the \emph{2}-\emph{adic order} or \emph{2}-\emph{adic valuation} of \emph{n}).

For proof of Theorem \ref{Resultants:Ft} see Subsection \ref{Proof:Resultants:Main:FT} on page \pageref{Proof:Resultants:Main:FT}.
For proof of Theorem \ref{main:2:lucas} see Subsection \ref{Proof:Resultants:Main:LT} on page \pageref{Proof:Resultants:Main:LT}.
For proof of Theorem \ref{Main:3:thm} see Subsection \ref{Proof:Resultants:LFT} on page \pageref{Proof:Resultants:LFT}.
For proof of Theorems \ref{Disc:Fibonacci:type} and  \ref{Disc:Lucas:type} see Subsection \ref{Proof:Disc:FT} on page \pageref{Proof:Disc:FT}.

\begin{theorem}\label{Resultants:Ft}
Let $T_{\Ft{}}= \left((-1)^{\eta\omega}\beta^{2\eta-\omega}\rho\right)^{\frac{(n-1)(m-1)}{2}}$
where  $n, m\in \mathbb{Z}_{>0}$. Then
		\[ \ress{\Ft{n}}{\Ft{m}}= 
		\begin{cases} 0 & \mbox{if }\;   \gcd(m,n)>1;\\
		T_{\Ft{}} & \text{otherwise.}
		\end{cases}
		\]
\end{theorem}

\begin{theorem}\label{main:2:lucas}
Let $T_{\Lt{}}=\alpha^{-\eta(n+m)} 2^{\eta \gcd(m,n)}\big((-1)^{\eta\omega}\beta^{2\eta-\omega}\rho\big)^{nm/2}$
where $m, n\in\mathbb{Z}_{>0}$. Then	
	\[ \ress{\Lt{m}}{\Lt{n}}= 
		\begin{cases} 0 &  \mbox{if }E_{2}(n)=E_{2}(m);\\
			   T_{\Lt{}} & \mbox{if }E_{2}(n)\not=E_{2}(m).
		\end{cases}
	\]	
\end{theorem}

\begin{theorem} \label{Main:3:thm}
Let $T_{\Lt{}\Ft{}}=2^{\eta \gcd(m,n)-\eta}\alpha^{\eta(1-m)}\left((-1)^{\eta\omega}\beta^{2\eta-\omega}\rho\right)^{(n(m-1))/2}$
where  $n, m \in \mathbb{Z}_{>0}$. Then
\[
\ress{\Lt{n}}{\Ft{m}}=
		\begin{cases}  0 			      & \mbox{if } E_{2}(n)<E_{2}(m);\\
				      T_{\Lt{}\Ft{}}	      & \mbox{if }  E_{2}(n)\ge E_{2}(m).
		\end{cases}
\]
\end{theorem}

\begin{theorem}\label{Disc:Fibonacci:type}  If $\deg(d)=1$, $g$ is a constant  and $d^{\prime}$ is the derivative of $d$, then
$$ \dis{\Ft{n}}= (-\rho)^{(n-2)(n-1)/2} (2 d^{\prime})^{n-1} n^{n-3}\beta^{(n-1)(n-3)}.$$
\end{theorem}

\begin{theorem} \label{Disc:Lucas:type} If $\deg(d)=1$, $g$ is a constant  and $d^{\prime}$ is the derivative of $d$, then
	$$\dis{\Lt{n}}= \left(-\rho\right)^{n(n-1)/2} 2^{n-1}\left(nd^{\prime}\right)^{ n}    \alpha^{2-2n}  \beta^{n(n-2)}.$$
\end{theorem}

\subsection{Corollaries. Resultants and discriminant of some known GFP sequences}

In this section we present corollaries of the main results.
Table \ref{corollary_Fibonacci} presents the resultants for some Fibonacci-type polynomials.
Table \ref{corollary_lucas} presents the resultants of some Lucas-type polynomials.
Table \ref{corollary_lucas_fibo} presents the resultants of two equivalent polynomials
(Lucas-type and its equivalent polynomial of Fibonacci-type).
Table \ref{Disc_Fibo_Lucas} gives the discriminants of familiar polynomials discussed
in this paper. In the first half of Table \ref{Disc_Fibo_Lucas} are the the Fibonacci-type polynomials
and in the second half of the table are the Lucas-type polynomials. 
Note that the derivatives for GFPs is given  in Table \eqref{Deriv:Gen:Falcon:Plaza} on page \pageref{Deriv:Gen:Falcon:Plaza}.

Note that the  following property can be used to find the discriminant  of a product of GFPs  (see \cite{Childs}). If $P$ and $Q$ are polynomials in $\mathbb{Q}[x]$, then  $\dis{PQ}=\dis{P}\dis{Q}\ress{P}{Q}$. 

\begin{table} [!ht]
\begin{center}\scalebox{1}{
\begin{tabular}{|l|l|l|l|} \hline
 		Polynomial      	 	 &  $\gcd(m,n)=1$ 							&  $\gcd(m,n)>1$                  \\ \hline \hline
 		Fibonacci            	 & $\ress{F_{m}}{ F_{n}}=1$ 					& $\ress{F_{m}}{ F_{n}}=0$              \\
 		Pell		         	 & $\ress{P_{m}}{ P_{n}}=2^{(m-1) (n-1)}$	  		& $\ress{P_{m}}{ P_{n}}=0$        \\
 		Fermat		         & $\ress{\Phi_{m}}{ \Phi_{n}}= (-18)^{(m-1) (n-1)/2}$ & $\ress{\Phi_{m}}{ \Phi_{n}}=0 $   \\
 		Chebyshev 2nd kind & $\ress{U_{m}}{ U_{n}}\,=(-4)^{(m-1) (n-1)/2}$  	& $\ress{U_{m}}{ U_{n}}\,=0  $        \\
  		Morgan-Voyce         &$\ress{B_{m}}{ B_{n}}=(-1)^{(m-1) (n-1)/2}$   		&$\ress{B_{m}}{ B_{n}}=0$              \\
 \hline
\end{tabular}}
\end{center}
\caption{Resultants of Fibonacci-type polynomials using Theorem \ref{Resultants:Ft}.} \label{corollary_Fibonacci}
\end{table}

\begin{table} [!ht]
	\begin{center}\scalebox{1}{
			\begin{tabular}{|l|l|l|} \hline
				Polynomial      	       & $E_{2}(m) \ne  E_{2}(n)$, $\delta= {\gcd (m,n)}$  					& $E_{2}(m)= E_{2}(n)$               \\ \hline \hline
				Lucas               	       & $\ress{D_{m}}{ D_{n}}=2^{\delta}$                         					& $\ress{D_{m}}{ D_{n}}= 0$                     \\
				Pell-Lucas-prime     & $\ress{Q_{m}^{\prime}}{ Q_{n}^{\prime}}= 2^{(m-1) (n-1)-1} 2^{\delta}$ 	& $\ress{Q_{m}^{\prime}}{ Q_{n}^{\prime}}=0$             \\
				Fermat-Lucas          & $\ress{\vartheta_{m}}{ \vartheta_{n}}=(-1)^{m n/2}18^{m n/2}\; 2^{\delta}$ & $\ress{\vartheta_{m}}{ \vartheta_{n}}\,\,=0$ \\
				Chebyshev 1st kind& $\ress{T_{m}}{ T_{n}}=(-1)^{\frac{m n}{2}} 2^{(m-1) (n-1)-1} 2^{\delta}$	& $\ress{T_{m}}{ T_{n}}\,\,=0$                 \\
				Morgan-Voyce        & $\ress{C_{m}}{ C_{n}}=(-1)^{\frac{m n}{2}}\; 2^{\delta} $				& $\ress{C_{m}}{ C_{n}}\,=0$                   \\
				\hline
		\end{tabular}}
	\end{center}
	\caption{Resultants of Lucas-type polynomials using Theorem \ref{main:2:lucas}} \label{corollary_lucas}
\end{table}

\begin{table} [!ht]
	\begin{center}\scalebox{.97}{
			\begin{tabular}{|l|l|l|} \hline
				Polynomials     	 	  		&  $E_{2}(n)\ge  E_{2}(m)$, $\delta= {\gcd (n,m)}$     			& $E_{2}(n)<  E_{2}(m)$             \\ \hline \hline
				Lucas, Fibonacci          		& $\ress{D_{n}}{F_{m}}=2^{\delta-1}$                 				& $\ress{D_{n}}{F_{m}}=0$           \\
				Pell-Lucas-prime, Pell      		& $\ress{Q_{n}^{\prime}}{P_{m}}=2^{(m-1) (n-1)}\; 2^{\delta-1}$		& $\ress{Q_{n}^{\prime}}{P_{m}}=0$       \\
				Fermat-Lucas, Fermat       	& $\ress{\vartheta_{n}}{\Phi_{m}}=(-18)^{ n(m-1)/2}\; 2^{\delta-1}$  	& $\ress{\vartheta_{n}}{\Phi_{m}}=0$\\
				Chebyshev both kinds 		& $\ress{T_{n}}{U_{m}}=(-1)^{n(m-1)/2} 2^{(m-1) (n-1)} 2^{\delta-1}$	&  $\ress{T_{n}}{U_{m}}=0$           \\
				Morgan-Voyce both types    	& $\ress{C_{n}}{B_{m}}=(-1)^{n(m-1)/2} \;2^{\delta-1}$			& $\ress{C_{n}}{B_{m}}=0$           \\
				\hline
		\end{tabular}}
	\end{center}	
\caption{Resultants of two equivalent  polynomials using Theorem \ref{Main:3:thm}.} \label{corollary_lucas_fibo}
\end{table}

\begin{table} [!ht]
\begin{center}\scalebox{1}{
\begin{tabular}{|l|l|l} \hline
 			Polynomial      	 	 &  Discriminants of GFP    						\\ \hline \hline
			 Fibonacci            	 & $\dis{F_{n}}=(-1)^{ (n-2) (n-1)/2} 2^{n-1} n^{n-3}$          \\
 			Pell		         	 & $\dis{P_{n}}=  (-1)^{(n-2) (n-1)/2} 2^{(n-1)^2} n^{n-3} $  \\
 			Fermat		          & $\dis{\Phi_{n}}= 2^{ n (n-1)/2} 3^{(n-1) (n-2)} n^{n-3} $  \\
 			Chebyshev 2nd kind  & $\dis{U_{n}}=  2^{(n - 1)^2} n^{n - 3}  $        			\\
 			Morgan-Voyce            &$\dis{B_{n}} = 2^{n-1} n^{n-3}   $            			\\ \hline  \hline
			Lucas               		 & $\dis{D_{n}}=(-1)^{n (n-1)/2} 2^{n-1} n^n$               	\\
			Pell-Lucas-prime        & $\dis{Q_{n}^{\prime}}=(-1)^{n (n-1)/2} 2^{(n-1)^2} n^n$  \\
			Fermat-Lucas         	 & $\dis{\vartheta_{n}}=2^{ (n-1) (n+2)/2} 3^{n (n-1)} n^n$  \\
			Chebyshev 1st kind   & $\dis{T_{n}}=2^{(n-1)^2} n^n$               				\\
			Morgan-Voyce           & $\dis{C_{n}}= 2^{n-1} n^n$                 				 \\
			\hline
\end{tabular}}
\end{center}
\caption{Discriminants  of GFP using Theorems \ref{Disc:Fibonacci:type} and \ref{Disc:Lucas:type}.} \label{Disc_Fibo_Lucas}
\end{table}

\section{Proofs of main results about the resultant of two GFP}
In this section we give the proofs of three main results presented in Section \ref{Main:Results}.

\subsection{Background and some known results}
Most of the results in this subsection are in \cite{florezHiguitaMuk2018}. Proposition \ref{main:lemma}
is a result that is in the proof of \cite[Proposition 6]{florezHiguitaMuk2018} therefore its proof is omitted. 

Recall that we use $\deg(P)$ and $\lc(P)$ to mean degree and leading  coefficient of a polynomial $P$.
In this paper we use   $\mathbb{Z}_{\ge 0}$ and  $\mathbb{Z}_{> 0}$
to mean the set of non-negative integers and  positive integers, respectively.
Recall that  $\beta$, $\lambda$, $\eta$, $\omega$, and $\rho$ are defined in \eqref{Degree:Lead:Def}
on page \pageref{Degree:Lead:Def} and that $d$, $g$ and $\alpha$ are defined on page \pageref{Fibonacci;general:LT}.
	
\begin{proposition} \label{main:lemma} Let $m$, $n$, $r$, and $q$ be positive integers. 
If $n=mq+r$, then there is  a polynomial $T $ such that $\Ft{n}=\Ft{m}T +g \Ft{mq-1}\Ft{r}$.

\end{proposition}
\begin{proposition} \label{Dic2} If $m$, $n$, $r$, and $q$ are positive integers, then 
	 \label{Dic2:Part2} if $r<m$, then there is a polynomial $T $ such that for $t=\left\lceil \frac{q}{2} \right\rceil$ this holds
			\[ \Lt{mq+r}=
			\left\{
			\begin{array}{ll}
			\Lt{m}T +(-1)^{m(t-1)+t+r}(g )^{(t-1)m+r} \Lt{m-r}, & \hbox{ if } $q$ \hbox{ is odd;} \\
			\Lt{m}T +(-1)^{(m+1)t}(g )^{mt} \Lt{r},             & \hbox{ if } $q$ \hbox{ is even}.
			\end{array}
			\right.
			\]		

\end{proposition}
	
\begin{proposition} \label{Fibonacci:Lucas:Identities}
If  $n$, $q$, and $r$ are nonnegative integers with $q>0$, then
\begin{enumerate}[(i)]
    \item \label{Fibonacci:Identities1}
		\[ \Ft{nq+r}=\left\{
		\begin{array}{ll} \alpha\Lt{n}\Ft{n(q-1)+r}-(-g)^n \Ft{n(q-2)+r}, & \mbox{if } q>1;\\
				      \alpha\Lt{n}\Ft{r}+(-g)^r\Ft{n-r}               & \mbox{if } q=1.
		\end{array}\right.
		\]
    \item \label{Lucas:Identities2}
		\[ 
		\alpha\Lt{nq+r}=\left\{
		\begin{array}{ll} 	(a-b)^2\Ft{n}\Ft{n(q-1)+r}+\alpha(-g)^n \Lt{n(q-2)+r} & \mbox{if } q>1;\\
				      		 	 (a-b)^2\Ft{n}\Ft{r}+\alpha(-g)^r \Lt{n-r}             & \mbox{if } q=1.
		\end{array} \right.
		\]
\end{enumerate}
\end{proposition}

\begin{proof}  We prove Part \eqref{Fibonacci:Identities1}, the proof of Part \eqref{Lucas:Identities2} is similar and it is omitted.
If $q=1$, then the proof follows from \cite[Proposition 3]{florezHiguitaMuk2018}. We now prove the case in which $q>1$.
Using Binet formulas \eqref{bineformulauno} and  \eqref{bineformulados} we obtain
$$\alpha\Lt{n}\Ft{n(q-1)+r}=\alpha\frac{(a^n+b^n)}{\alpha}\frac{(a^{n(q-1)+r}-b^{n(q-1)+r})}{a-b}.$$
Expanding and simplifying we have
\[\alpha\Lt{n}\Ft{n(q-1)+r}=\Ft{nq+r}+(ab)^{n}\frac{a^{n(q-2)+r}-b^{n(q-2)+r}}{a-b}=\Ft{nq+r}+(-g)^n\Ft{n(q-2)+r}.\]

Solving this equation  for $\Ft{nq+r}$ we have  $\Ft{nq+r}=\alpha\Lt{n}\Ft{n(q-1)+r}-(-g)^n\Ft{n(q-2)+r}$. This completes the proof.
\end{proof}

\begin{lemma}\label{Propierties:FT} Let $\beta=\lc(d)$ and $\eta = \deg(d)$. Then
		\begin{enumerate}[(i)]
			\item\label{degft} $\deg \left(\Ft{k} \right)=\eta(k-1)$ and  $\lc\left(\Ft{k}\right)=\beta^{k-1}$.
			\item \label {Propierties:LT}  $\deg\left(\Lt{n}\right)=\eta n$ and   $\lc(\Lt{n})=\beta^{n}/\alpha$.
		\end{enumerate}
	\end{lemma}
	
	\begin{proof} We use mathematical induction to prove all parts. We prove Part \eqref{degft}. Let $P(k)$ be the statement:
    $$\deg\left(\Ft{k}\right)=\eta(k-1) \text{ for every } k\ge 1.$$
	The basis step, $P(1)$, is clear, so we suppose that $P(k)$ is true for $k=t$, where  $t>1$. Thus, we suppose that
	$\deg\left(\Ft{t}\right)=\eta(t-1)$ and we prove $P(t+1)$.
	We know that $\deg(\Ft{n})\ge \deg(\Ft{n-1})$ for $n\ge1$. This, $\deg\left(d \right)>\deg\left(g \right)$, and \eqref {Fibonacci;general:FT} imply
	 $$\deg\left(\Ft{t+1}\right)=\deg\left(d \Ft{t}\right)= \deg(d )+\deg\left(\Ft{t}\right)=\eta+\eta(t-1)=\eta t.$$
		
	We  now prove the second half of Part \eqref{degft}. Let $Q(k)$ be the statement: $$\lc(\Ft{k})= \beta^{k-1} \text{ for every } k\ge 1.$$
	The basis step, $Q(1)$, is clear, so we suppose that $Q(k)$ is true for $k=t$, where  $t>1$. Thus, we suppose that
	$\lc \left(\Ft{t}\right)= \beta^{t-1}$ and we prove $Q(t+1)$. We know that $\deg(\Ft{n})\ge \deg(\Ft{n-1})$ for $n\ge1$. This,
	$\deg\left(d \right)>\deg\left(g \right)$, and \eqref {Fibonacci;general:FT} imply
    $\lc(\Ft{t+1})=\lc(d)\lc(\Ft{t})=\lc(d) \beta^{t-1}= \beta\beta^{t-1}=\beta^{t}$.
	
	We prove Part \eqref{Propierties:LT}. Let $H(n)$ be the statement:  $\deg\left(\Lt{n}\right)=\eta n$ for every $n>0$. 
	It is easy to see that  $H(1)$ is true.  Suppose that $H(n)$ is true for some $n=k>1$. Thus, suppose that $\deg\left(\Lt{k}\right)=\eta k$ and 
	we prove $H(k+1)$. Since $\Lt{k+1}=d\Lt{k}+g\Lt{k-1}$ and $\deg\left(d\right)>\deg\left(g\right)$,
	we have $\deg\left(\Lt{k+1}\right)=\deg \left(d\right)+\deg\left(\Lt{k}\right)=\eta+\eta k=\eta(k+1)$. This proves the first half of Part \eqref{Propierties:LT}.
		
	We  now prove the second half of Part \eqref{Propierties:LT}.  Let $N(n)$ be the statement:  $\lc(\Lt{n})=\beta^{n}$ for every $n>0$. 
	(for simplicity we suppose that $\alpha =1$).  It is easy to verify that $\lc(\Lt{1})=\beta^{1}$. Suppose that $N(n)$ is true for some $n=k>1$. Thus, suppose that 
		$\lc(\Lt{k})=\beta^{k}$.  Since  $\Lt{k+1}=d\Lt{k}+g\Lt{k-1}$ and  $\deg\left(d\right)>\deg\left(g\right)$, we have
		$\lc(\Lt{k+1})=\lc(d)\lc(\Lt{k})$. This and the inductive hypothesis imply that   $\lc(\Lt{k+1})=\beta \beta^{k}=\beta^{k+1}$.
	\end{proof}
	
Proposition \ref{second:main:thm} plays an important role in this paper. This in connection with Lemma \ref{Properties:res} Part \eqref{Fundamental:propierty:resultant} gives criterions to determine whether or not the resultant of two GFPs is equal to zero (see 
Propositions \ref{Cor:1},  \ref{Cor:2}, and \ref{cor3}).
	
Recall that definition of  $E_{2}(n)$ was given in Section \ref{Main:Results} on page \pageref{adic:order}.

\begin{proposition}[\cite{florezHiguitaMuk2018}] \label{second:main:thm}
		If $m, n\in\mathbb{Z}_{>0}$ and $\delta=\gcd(m,n)$, then these hold
\begin{enumerate}[(i)]
\item  \label{gcd:Two:Fib}  $\gcd(\Ft{m},\Ft{n})=1$ if and only if $\delta=1$. 
\item \label{gcd:two:lucas}
		\[ \gcd \left(\Lt{m},\Lt{n} \right)=
		\begin{cases} \Lt{\delta}  & \mbox{if }\;   E_{2}(m)= E_{2}(n);\\
				      \gcd \left(\Lt{\delta},\Lt{0} \right)   & \mbox{otherwise}.
		\end{cases}
		\]
	
\item \label{gcd:LnFm}
		\[ \gcd(\Lt{n},\Ft{m})=
		\begin{cases} \Lt{\delta} & \mbox{if }E_2(m)>E_{2}(n);\\
				     1 & \mbox{otherwise.}
		\end{cases}
		\]	
\end{enumerate}
\end{proposition}

\subsection{Properties of the resultant and some resultants of GFP of Fibonacci-type}
In this section we give some classic properties of the resultant and some results needed to prove Theorem \ref{Resultants:Ft}.

Let $f $ and $h$ be polynomials where $a_n=\lc(f )$, $b_m=\lc(h )$, $n=\deg(f )$ and $m=\deg(h)$. The resultant of two
polynomials is defined using the Sylvester determinant (See for example \cite{Akritas,Jacobs,Sylvester}).
The next properties are well known and may be found in \cite{Basu}. For a complete development of the theory of the
resultant of polynomials see \cite{Gelfand:Discriminants}.
Most of the parts Lemma \ref{Properties:res} can be found in \cite{Basu, Serge:algebra}. Note that if $k$ is a constant, then
$\ress{k}{f}=\ress{f}{k}=k^{\deg(f)}$.

\begin{lemma}[\cite{Jacobs}]\label{Properties:res}
Let $f $, $h$, $p$, and $q$ be polynomials in $\mathbb{Q}[x]$. If $n=\deg(f )$, $m=\deg(h)$, $a_n=\lc(f )$  and $b_m=\lc(h)$, then these hold
\begin{enumerate}[(i)]
\item \label{Rbasic-1} $\ress{f}{h}  = (-1)^{nm}\ress{h}{f}$,
\item \label{Rbasic-2} $\ress{f}{ph} =\ress{f}{p}\ress{f}{h}$,
\item \label{Rbasic-2add} $\ress{f}{p^n} =\ress{f}{p}^n$,
\item \label{Rbasic-3}  if $G =f q +h$ and $r=\deg(G)$, then $\ress{f}{G}  =a_{n}^{r-m}\ress{f}{h}$,
\item \label{Fundamental:propierty:resultant} The $\ress{f}{h}=0$ if and only if $f$ and $h$ have a common divisor of positive degree.
\end{enumerate}
\end{lemma}

\begin{lemma} \label{lemma:g:factor} For $m$ and $n$ in $\mathbb{Z}_{\ge 0}$ these hold
\begin{enumerate} [(i)]
\item  \label{resultantdg} $\ress{g }{\Ft{n}}=\rho^{n-1}$,
\item  \label{gfactor:FIbo:Fibo} $\ress{\Ft{m}}{g\Ft{n}}=(-1)^{\omega\eta(m-1)}\rho^{m-1}\ress{\Ft{m}}{\Ft{n}}$,
\item  \label{gfactor:Lucas:Lucas} $\ress{\Lt{m}}{g\Lt{n}}=(-1)^{\omega\eta m}\rho^{m}\ress{\Lt{m}}{\Lt{n}}$.
\end{enumerate}
\end{lemma}

\begin{proof}  We prove Part \eqref{resultantdg} using mathematical induction. Let $P(n)$ be the statement:
    $$\ress{g }{\Ft{n}}=\rho^{n-1} \text{ for every } n\ge 1.$$
	Since $\Ft{1}=1$ the basis step, $P(1)$, is clear. Suppose that $P(n)$ is true for $n=k$, where  $k>1$. Thus, suppose that
	$\ress{g }{\Ft{k}}=\rho^{k-1}$, and we prove $P(k+1)$. From \eqref {Fibonacci;general:FT} and Lemma \ref{Properties:res}
	Parts \eqref{Rbasic-2} and \eqref{Rbasic-3}  we have
		\[
			\ress{g}{\Ft{k+1}}=\ress{g}{d\Ft{k}+g\Ft{k-1}}=\lambda^{\eta k-(\eta+\eta(k-1))}\ress{g}{d}\ress{g}{\Ft{k}}.
		\]
		This and $P(k)$ imply
		$
			\ress{g}{\Ft{n}}=\ress{g}{d}\ress{g}{\Ft{n-1}}= \rho^{n-1}.
		$

 We prove Part \eqref{gfactor:FIbo:Fibo}, the proof of Part \eqref{gfactor:Lucas:Lucas} is similar and it is omitted.
From Lemma \ref{Properties:res} Part \eqref{Rbasic-1} and Part \eqref{Rbasic-2}, and Lemma \ref{Propierties:FT} Part \eqref{degft}, we have
\begin{eqnarray*}
\ress{\Ft{m}}{g\Ft{n}}&=&\ress{\Ft{m}}{g}\ress{\Ft{m}}{\Ft{n}}\\
					  &=&(-1)^{\omega\eta(m-1)}\ress{g}{\Ft{m}}\ress{\Ft{m}}{\Ft{n}}\\
                      &=&(-1)^{\omega\eta(m-1)}\rho^{m-1}\ress{\Ft{m}}{\Ft{n}}.
\end{eqnarray*}
This completes the proof.
\end{proof}

\begin{proposition}\label{Cor:1} Let $m, n\in \mathbb{Z}_{>0}$. Then  $\gcd \left(m,\,n \right)=1$  if and only if   $\ress{\Ft{m}}{\Ft{n}}\neq 0$.
	\end{proposition}
		
\begin{proof}  Proposition \ref{second:main:thm} Part \eqref{gcd:Two:Fib}  and Lemma \ref{Properties:res} Part
	\eqref{Fundamental:propierty:resultant} imply that $\ress{\Ft{m}}{\Ft{n}}\neq 0$ if and only if $\gcd \left(m,\,n \right)=1$.
\end{proof} 
		
	\begin{lemma} \label{lemma:Resultant:FibonaciT} If  $m$, $n$ and $q$ are positive integers and $\Ft{t}\ne 0$ for every $t$, then these hold
	 	\begin{enumerate} [(i)]
	 		\item \label{resnn-1} $ \ress{\Ft{n}}{\Ft{n-1}}=\left((-1)^{\omega\eta}\beta^{2\eta-\omega}\rho\right)^{(n-2)(n-1)/2}$,
	 		\item \label{resultant:m:mq-1} $\ress{\Ft{m}}{\Ft{mq-1}}= \left((-1)^{\eta\omega}\beta^{2\eta-\omega}\rho \right)^{(m-1)(mq-2)/2}$.
		 \end{enumerate}
	\end{lemma}
	
	\begin{proof} We prove all parts by mathematical induction.		
Proof of Part \eqref{resnn-1}. Let $Q(n)$ be the statement:
		 $$\ress{\Ft{n}}{\Ft{n-1}}=\left((-1)^{\omega\eta}\beta^{2\eta-\omega}\rho\right)^{(n-2)(n-1)/2} \text{ for every } n\ge 2.$$
		Since $\Ft{1}=1$, the basis step, $Q(2)$, is clear. Suppose that $Q(n)$ is true for $n=k-1$, where  $k> 2$.
		Thus, suppose that $\ress{\Ft{k-1}}{\Ft{k-2}}=\left((-1)^{\omega\eta}\beta^{2\eta-\omega}\rho\right)^{(k-3)(k-2)/2}$. We prove $Q(k)$.	
		Using Lemma \ref{Propierties:FT} Part \eqref{degft}, Lemma \ref{Properties:res} Part \eqref{Rbasic-1}, and \eqref{Rbasic-3}, we get
		\begin{eqnarray*}
			\ress{\Ft{n}}{\Ft{n-1}}
			&=&(-1)^{\eta^2(n-1)(n-2)}\ress{\Ft{n-1}}{\Ft{n}}\\&=&\ress{\Ft{n-1}}{d \Ft{n-1}+g\Ft{n-2}}.
		\end{eqnarray*}
		This, Lemma \ref{Propierties:FT} Part \eqref{degft} and Lemma \ref{lemma:g:factor} Part \eqref{gfactor:FIbo:Fibo} imply
		\begin{eqnarray*}
			\ress{\Ft{n}}{\Ft{n-1}}&=&(\beta^{n-2})^{\eta(n-1)-(\omega+\eta(n-3))}\ress{\Ft{n-1}}{g\Ft{n-2}}\\
			&=&(-1)^{\omega\eta(n-2)}\beta^{(n-2)(2\eta-\omega)}\rho^{n-2}\ress{\Ft{n-1}}{\Ft{n-2}}.
		\end{eqnarray*}
Simplifying we have
		\begin{equation*}
			\ress{\Ft{n}}{\Ft{n-1}}=\left((-1)^{\omega\eta}\beta^{2\eta-\omega}  \rho \right)^{n-2}\ress{\Ft{n-1}}{\Ft{n-2}}.
		\end{equation*}
		This and $Q(k-1)$ give
		\begin{eqnarray*}
			\ress{\Ft{n}}{\Ft{n-1}}&=& \left((-1)^{\omega\eta}\beta^{2\eta-\omega}  \rho \right)^{n-2} \left(\beta^{2\eta-\omega}(-1)^{\omega\eta}\rho \right)^\frac{(n-3)(n-2)}{2}\\
			&=& \left((-1)^{\omega\eta}\beta^{2\eta-\omega} \rho \right)^\frac{(n-2)(n-1)}{2}.
		\end{eqnarray*}
		
	Proof of Part \eqref{resultant:m:mq-1}.  Let $W(q)$ be the statement:  for a fixed integer $m$ this holds
		 $$\ress{\Ft{m}}{\Ft{mq-1}}= \left((-1)^{\eta\omega}\beta^{2\eta-\omega}\rho \right)^{(m-1)(mq-2)/2} \text{ for every } q\ge 1.$$
		 From Lemma \ref{lemma:Resultant:FibonaciT} Part \eqref{resnn-1} it follows that $W(1)$ is true.  Suppose that $W(q)$ is true for
		 $q=k$, where  $k> 1$. Thus, suppose that $\ress{\Ft{m}}{\Ft{mk-1}}= \left((-1)^{\eta\omega}\beta^{2\eta-\omega}\rho \right)^{(m-1)(mk-2)/2}$.
		 We prove $W(k+1)$. From Proposition \ref{main:lemma} we know that there is a polynomial $T $ such that
		 $\Ft{m(k+1)-1}=\Ft{m}T +g \Ft{km-1}\Ft{m-1}$.
		 This, Lemma \ref{Properties:res} Part \eqref{Rbasic-2} and Part \eqref{Rbasic-3} and Lemma \ref{lemma:g:factor} Part \eqref{gfactor:FIbo:Fibo} imply
		\begin{eqnarray*}
			\ress{\Ft{m}}{\Ft{m(k+1)-1}}&=&\ress{\Ft{m}}{\Ft{m}T +g \Ft{km-1}\Ft{m-1}}\\
			&=&(\beta^{m-1})^{\eta(mk+m-2)-(\omega+\eta(mk+m-4))}\ress{\Ft{m}}{g \Ft{km-1}\Ft{m-1}}\\
			&=&\beta^{(m-1)(2\eta -\omega)}\ress{\Ft{m}}{g \Ft{km-1}\Ft{m-1}}\\
			&=&(-1)^{\omega\eta(m-1)}\beta^{(m-1)(2\eta -\omega)}\rho^{m-1}\ress{\Ft{m}}{\Ft{m-1}}  \ress{\Ft{m}}{\Ft{km-1}}.
		\end{eqnarray*}			
	From this, Part \eqref{resnn-1}, Lemma \ref{lemma:g:factor} Part \eqref{resultantdg}, and $W(k)$ we conclude
		\begin{eqnarray*}
			\ress{\Ft{m}}{\Ft{mk+m-1}}&=& \left((\beta^{m-1})^{2\eta -\omega}\ress{\Ft{m}}{g } \right)^k \ress{\Ft{m}}{\Ft{m-1}} ^{k+1}\\
			&=& \left((-1)^{\omega\eta(m-1)}\rho^{m-1}(\beta^{m-1})^{2\eta -\omega} \right)^k  \left((-1)^{\omega\eta}\beta^{2\eta-\omega}\rho \right)^{\frac{(m-2)(m-1)(k+1)}{2}}\\
			&=& \left((-1)^{\omega\eta}\beta^{2\eta -\omega} \rho\right)^{k(m-1)}  \left((-1)^{\omega\eta}\beta^{2\eta-\omega}\rho \right)^{\frac{(m-2)(m-1)(k+1)}{2}}.
		\end{eqnarray*}
		Simplifying the  last expression, we have that  $$\ress{\Ft{m}}{\Ft{mk+m-1}}=
		 \left((-1)^{\eta\omega}\beta^{2\eta-\omega}\rho \right)^{\frac{(m-1)(mk+m-2)}{2}}.
		$$  This completes the proof.
	\end{proof}
\subsection{Proof of Theorem \ref{Resultants:Ft} }\label{Proof:Resultants:Main:FT}
We now prove the resultant of two Fibonacci-type polynomials.

\begin{proof}[Proof of Theorem \ref{Resultants:Ft}]  Let $A$ be the set of all $i \in \mathbb{Z}_{> 0}$  such that for  every $j \in \mathbb{Z}_{> 0}$ this holds
\begin{equation}\label{resulFnm}
\ress{\Ft{i}}{\Ft{j}}=
		\begin{cases} 0 & \mbox{ if }\;   \gcd(m,n)>1;\\
				 \left((-1)^{\eta\omega}\beta^{2\eta-\omega}\rho\right)^{\frac{(i-1)(j-1)}{2}}& \text{ otherwise.}
		\end{cases}
\end{equation}

Since $1 \in A$, we have that $A\ne \emptyset$. The following claim completes the proof of the Theorem.

\noindent  {\bf Claim.} $A=\mathbb{Z}_{> 0}$.

\noindent \emph{Proof of Claim.}
Suppose $B:= \mathbb{Z}_{> 0} \setminus A$ is a non-empty set. Let $n\ne 1$ be the least element of $B$. So, there is
$h\in \mathbb{Z}_{> 0}$ such that $\ress{\Ft{n}}{\Ft{h}}$ does not satisfy Property \eqref{resulFnm} (if $m=n$, then $\ress{\Ft{i}}{\Ft{j}}=0$).
Let $m$ be the least element of the non-empty set   $H=\{h \in \mathbb{Z}_{>0}\mid \ress{\Ft{n}}{\Ft{h}} \text{ does not satisfy } \eqref{resulFnm}\}$.
Note that Proposition \ref{Cor:1} and  \eqref{resulFnm} imply that  that  $\gcd(m,n)=1$. We now consider two cases.

 \textbf{Case $m<n$.} Since $n$ is the minimum element of $B$, $m \in A$. Either $m$ or $n$ is odd, because  $\gcd(m,n)=1$.
 We know, from Lemma \ref{Propierties:FT} Part \eqref{degft},  that $\deg(\Ft{m})=\eta(m-1)$. This implies that  $\ress{\Ft{n}}{\Ft{m}}=\ress{\Ft{m}}{\Ft{n}}$.
Since  $m\in A$, we have that \eqref{resulFnm} holds for $j\in \mathbb{Z}_{> 0}$, in particular  \eqref{resulFnm} holds when  $j=n$. That is a contradiction.

\textbf{Case $n<m$.}  The Euclidean algorithm and $\gcd(m,n)=1$ guarantee that there are  $q, r \in \mathbb{Z}$ such that $m=nq+r$ with  $0<r<n$.
We now can proceed analogously to the proof of Lemma \ref{lemma:Resultant:FibonaciT} Part \eqref{resultant:m:mq-1}. From the Euclidean algorithm, Proposition \ref{main:lemma} and
Lemma \ref{Properties:res} Part  \eqref{Rbasic-3} we have
		\begin{eqnarray*}
			\ress{\Ft{n}}{\Ft{m}}
			&=& \ress{\Ft{n}}{\Ft{nq+r}}\\
			&=& \ress{\Ft{n}}{\Ft{n}T+g\Ft{nq-1}\Ft{r}}\\
			&=& (\beta^{n-1})^{\eta(m-1)-(\omega+\eta(nq-2+r-1))}\ress{\Ft{n}}{g\Ft{nq-1}\Ft{r}}.
		\end{eqnarray*}
This, Lemma \ref{Properties:res} Part  \eqref{Rbasic-2} and Lemma  \ref{lemma:g:factor} Part \eqref{gfactor:FIbo:Fibo} imply 		
\begin{eqnarray}
			\ress{\Ft{n}}{\Ft{m}}	&=& \beta^{(n-1)(2\eta-\omega)}\ress{\Ft{n}}{g\Ft{r}}\ress{\Ft{n}}{\Ft{nq-1}}\nonumber\\
			&=&(-1)^{\omega\eta(n-1)}\beta^{(n-1)(2\eta-\omega)}\rho^{n-1} \big((-1)^{\eta\omega}\beta^{2\eta-\omega}\rho\big)^{\frac{(n-1)(nq-2)}{2}} \ress{\Ft{n}}{\Ft{r}}\nonumber\\
			&=& \big((-1)^{\eta\omega}\beta^{2\eta-\omega}\rho\big)^{\frac{(n-1)nq}{2}} \ress{\Ft{n}}{\Ft{r}}\label{resultFnm1}.
 \end{eqnarray}
Since  $\gcd(n,m)=\gcd(n,r)=1$, either $n$ or $r$ is odd. So, $\ress{\Ft{n}}{\Ft{r}}=\ress{\Ft{r}}{\Ft{n}}$. It is easy to verify that $r\in A$, because
$r<n$ and  $n$ is the minimum element of $B$. Set $j=n$, so
$\ress{\Ft{r}}{\Ft{n}}=\left((-1)^{\eta\omega}\beta^{2\eta-\omega}\rho\right)^{(n-1)(r-1)/2}$. This and \eqref{resultFnm1} imply that
 \begin{eqnarray*}
			\ress{\Ft{n}}{\Ft{m}}&=& \left((-1)^{\eta\omega}\beta^{2\eta-\omega}\rho\right)^{\frac{(m-1)(n-r)}{2}} \left((-1)^{\eta\omega}\beta^{2\eta-\omega}\rho\right)^{\frac{(r-1)(m-1)}{2}}\\
			&=& \left((-1)^{\eta\omega}\beta^{2\eta-\omega}\rho\right)^{\frac{(n-1)(m-1)}{2}}.
		\end{eqnarray*}
That is a contradiction. This implies that $A=\mathbb{Z}_{> 0}$.
\end{proof}

\subsection{Some resultants of GFP of Lucas-type}
Recall that a  GFP of  Lucas-type is a polynomial sequence such that $\Lt{0}\in\{1,2\}$, $\Lt{1}=2^{-1}\Lt{0}d$, and  $\Lt{n}=d\Lt{n-1}+g\Lt{n-2}$ for $n>1$.
	
Note that if we take the particular case of the Lucas-type sequence $\Lt{n}$ in which $\Lt{0}=1$ and  $\Lt{1}=2^{-1}d$, then using the initial conditions
we define a new Lucas-type sequence as follows:
Let $\overline{\Lt{0}}=2\Lt{0}$, $\overline{\Lt{1}}=2\Lt{1}=d$ and  $\overline{\Lt{n}}=d\overline{\Lt{n-1}}+g\overline{\Lt{n-2}}$ for  $n>1$.  It is easy to
verify that $\overline{\Lt{n}}=2\Lt{n}$ for $n\ge 0$. Therefore, to find the resultant of a polynomial of Lucas-type $\Lt{n}$, it is enough to find the
resultant for $\Lt{n}$ in which  $\Lt{0}=2$.
	
The following Proposition is actually a corollary of Proposition  \ref{second:main:thm} Part \eqref{gcd:two:lucas}.

\begin{proposition}\label{Cor:2} Let $m, n\in \mathbb{Z}_{>0}$. Then $E_{2}(n)= E_{2}(m)$ if and only if  $\ress{\Lt{m}}{\Lt{n}}= 0$.
\end{proposition}

\begin{proof}  Suppose that $E_{2}(m)=E_{2}(n)$. Therefore,  Proposition  \ref{second:main:thm} Part \eqref{gcd:two:lucas},  Lemma \ref{Properties:res} Part \eqref{Fundamental:propierty:resultant} and the fact that $\deg(\Lt{\gcd(m,n)})>1$, imply that  if $E_{2}(m)=E_{2}(n)$, then
$\ress{\Lt{m}}{\Lt{n}}=0$.

From Proposition  \ref{second:main:thm} Part \eqref{gcd:two:lucas} we have that if $E_{2}(m)\ne E_{2}(n)$, then $\gcd(\Lt{m},\Lt{n})=1 \text{ or } 2$.  For the other implication
we suppose that $E_{2}(n) \not=E_{2}(m)$ and that $\ress{\Lt{m}}{\Lt{n}}=0$. This and Lemma \ref{Properties:res} Part \eqref{Fundamental:propierty:resultant}
imply that  $\deg(\gcd(\Lt{n},\Lt{m}))\ge 1$. That is a contradiction.
\end{proof}

\begin{lemma}\label{Lemmas:for:Lucas}  If $n\in\mathbb{Z}_{>0}$ and $\Lt{0}=2$, then these hold
	
\begin{enumerate}[(i)]	
	\item \label{resultant:dgL}  $\ress{g}{\Lt{n}}=\rho^n$,
	\item \label{resultant:RdL}
		\[ \ress{\Lt{1}}{\Lt{n}}=
		\begin{cases} 0 & \mbox{if $n$  is odd;}\\
				2^{\eta}\left((-1)^{\eta \omega }\beta^{2\eta-\omega} \rho\right)^{\frac{n}{2}} & \mbox{if $n$ is even. }
		\end{cases}
		\]	
\end{enumerate}
\end{lemma}
	
	\begin{proof}
	We prove Part \eqref{resultant:dgL} by  mathematical induction on $n$. Since $\ress{g}{\Lt{1}}=\ress{g}{d}=\rho$, it holds that the result
	is true for  $n=1$. Suppose that for some integer $n=k>1$,  $\ress{g}{\Lt{k}}=\rho^k$ holds.
	From \eqref{Fibonacci;general:LT} and Lemma \ref{Properties:res} Parts  \eqref{Rbasic-2} and \eqref{Rbasic-3} we have
		$$
			\ress{g}{\Lt{k+1}}=\ress{g}{d\Lt{k}+g\Lt{k-1}} = \alpha^{\eta (k+1)-\eta (k+1)}\ress{g}{d}\ress{g}{\Lt{k}}.
		$$
	This and the inductive hypothesis imply that
		$$ \ress{g}{\Lt{k+1}}=\ress{g}{d}\ress{g}{\Lt{k}}=\ress{g}{d}^{k+1},$$
	which is our claim.	
	
We prove Part \eqref{resultant:RdL} by induction on $n$.  Let $Q(n)$ be the statement:
		\[ \ress{\Lt{1}}{\Lt{n}}=
		\begin{cases} 0 & \mbox{if $n$ is odd;}\\
				2^{\eta}\left((-1)^{\eta \omega }\beta^{2\eta-\omega} \rho\right)^{\frac{n}{2}} & \mbox{if  $n$ is even. }
		\end{cases}
		\]
Since $\ress{\Lt{1}}{\Lt{1}}=0$, $Q(1)$ holds. Note that $\Lt{1}=(p_{0}/2)d=d$. This and Lemma \ref{Properties:res} Parts \eqref{Rbasic-2} and \eqref{Rbasic-3} imply that
\[
			\ress{\Lt{1}}{\Lt{2}}=\ress{\Lt{1}}{d\Lt{1}+g\Lt{0}}=\beta^{2\eta-\omega}\ress{d}{2g}=2^{\eta} (-1)^{\eta\omega}\beta^{2\eta-\omega}\rho.
\]
This proves $Q(2)$.

Suppose that $Q(k-2)$ and $Q(k-1)$ is true and we prove $Q(k)$. Note that if $k$ is odd by Proposition \ref{Cor:2} we have that $\ress{\Lt{1}}{\Lt{n}}=0$. We suppose that $k$ is even. Lemma \ref{Properties:res}  Parts \eqref{Rbasic-2} and \eqref{Rbasic-3}, $\Lt{1}=d$, and Lemma \ref{Propierties:FT} Part \eqref{Propierties:LT}, and \eqref{bineformulados} imply
	\begin{eqnarray*}
		\ress{\Lt{1}}{\Lt{k}}&=&\ress{d}{d\Lt{k-1}+g\Lt{k-2}}\\
		&=&\beta^{2\eta-\omega}\ress{d}{g\Lt{k-2}}\\
		&=&\beta^{2\eta-\omega}\ress{d}{g}\ress{d}{\Lt{k-2}}\\
		&=&\left( (-1)^{\eta\omega}\beta^{2\eta-\omega}\rho\right) \ress{\Lt{1}}{\Lt{k-2}}.
	\end{eqnarray*}
Note that $k-2$ and $k$ have the same parity. This and $Q(k-2)$ imply that
	\[
		\ress{\Lt{1}}{\Lt{k}}=\left( (-1)^{\eta\omega}\beta^{2\eta-\omega}\rho\right)2^{\eta}\left((-1)^{\eta \omega }\beta^{2\eta-\omega}\rho\right)^{\frac{k-2}{2}} =2^{\eta}\left((-1)^{\eta \omega }\beta^{2\eta-\omega} \rho\right)^{\frac{k}{2}}.
\]
This proves $Q(k)$.
\end{proof}
\subsection{Proof of Theorem \ref{main:2:lucas}}\label{Proof:Resultants:Main:LT}
We now prove the resultant of two Lucas-type polynomials.
	
\begin{proof}[Proof of Theorem \ref{main:2:lucas}] We consider two cases: $\Lt{0}=2$ and $\Lt{0}=1$. We first prove the case $\Lt{0}=2$.
Let $A=\{i \in \mathbb{Z}_{>0}\mid \forall j \in \mathbb{Z}_{>0},  \text{Property \eqref{resultLnm1} holds for } \ress{\Lt{i}}{\Lt{j}}  \}$.
\begin{equation}\label{resultLnm1}
			 \ress{\Lt{i}}{\Lt{j}}=
			 	\begin{cases} 
				0 &  \mbox{if }E_{2}(i)=E_{2}(j);\\
				2^{\eta \gcd(i,j)}\left((-1)^{\eta\omega}\beta^{2\eta-\omega}\rho\right)^{ij/2}  &\mbox{if }E_{2}(i)\not=E_{2}(j).
		\end{cases}
\end{equation}	
From Lemma \ref{Lemmas:for:Lucas} Part \eqref{resultant:RdL} we have that $i=1 \in A$.  So, $A\ne \emptyset$.
The following claim completes the proof part $\Lt{0}=2$.

\noindent {\bf Claim.}  $A=\mathbb{Z}_{> 0}$.

\noindent \emph{Proof of Claim.} Suppose $B:=\mathbb{Z}_{> 0} \setminus A$ is a non-empty set.  Let $n\ne 1$ be the least element of $B$.
So, there is $h\in \mathbb{Z}_{> 0}$ such that $\ress{\Lt{n}}{\Lt{h}}$ does not satisfy Property \eqref{resultLnm1}. Let $m$ be the least element of the
non-empty set   $H=\{h \in \mathbb{Z}_{>0}\mid \ress{\Lt{n}}{\Lt{h}} \text{ does not satisfy  Property } \eqref{resultLnm1}\}$.  Note that
if  $E_{2}(n)=E_{2}(m)$, then $\ress{\Lt{n}}{\Lt{m}}=0$ (by Proposition  \ref{Cor:2}). That is a contradiction by the definition of $H$. Therefore,
we have that $E_{2}(n)\not=E_{2}(m)$. So,  $n\not=m$ where at least one of them is even. This implies that
 $\ress{\Lt{m}}{\Lt{n}}=(-1)^{\eta^2mn}\ress{\Lt{n}}{\Lt{m}}=\ress{\Lt{n}}{\Lt{m}}$.  Therefore, $\ress{\Lt{m}}{\Lt{n}}$ does not satisfy $\eqref{resultLnm1}$.
 So, $m \notin A$. Since $n\ne m$ is the least element of $B$, we have $m>n$. From the Euclidean algorithms we know that there are
 $q, r \in  \mathbb{Z}_{\ge 0}$  such that $m=nq+r$ with $0\le r<n$.

 We now proceed by cases over $q$.

\textbf{Case $q$ odd} Suppose that $q=2t-1$. Note that $t= \lceil q/2\rceil$ and that $(m-n+r)/2=(t-1)n+r$.  Since  $E_{2}(n)\ne E_{2}(m)$,  $r\ne0$ and $n(n-r)$ is even. This, Proposition \ref{Dic2} for odd case, and Lemma \ref{Properties:res} Part  \eqref{Rbasic-2} and Part \eqref{Rbasic-3} imply that $\ress{\Lt{n}}{\Lt{m}}$ equals
\begin{eqnarray*}
\ress{\Lt{n}}{\Lt{nq+r}}&=&\ress{\Lt{n}}{\Lt{n}T+(-1)^{t(n+1)+r-n}g^{(t-1)n+r}\Lt{n-r}}\\
 				&=&\ress{\Lt{n}}{\Lt{n}T+(-1)^{t(n+1)+r-n}g^{\frac{m-n+r}{2}}\Lt{n-r}}\\
 				&=&(\beta^n)^{\eta m-\frac{\omega(m-n+r)}{2}-\eta(n-r) }\ress{\Lt{n}}{(-1)^{t(n+1)+r-n}g^{\frac{m-n+r}{2}}\Lt{n-r}}\\
 				&=&\beta^{\frac{n(m-n+r)(2\eta-\omega)}{2}}\ress{\Lt{n}}{(-1)^{t(n+1)+r-n}g^{\frac{m-n+r}{2}}\Lt{n-r}}.
\end{eqnarray*}
Note that $\ress{\Lt{n}}{(-1)^{t(n+1)+r-n}}=(-1)^{\eta n(t(n+1)+r-n)}=(-1)^{\eta n(r-n)}=1$.
This, Lemma \ref{Properties:res} Parts \eqref{Rbasic-2add}  and \eqref{Rbasic-3} and Lemma \ref{Lemmas:for:Lucas}  Part \eqref{resultant:dgL} imply
 	\begin{eqnarray*}
 \ress{\Lt{n}}{\Lt{m}}&=&\beta^{\frac{n(m-n+r)(2\eta-\omega)}{2}}\ress{\Lt{n}}{g^{\frac{m-n+r}{2}}\Lt{n-r}}\\
 			      &=&\beta^{\frac{n(m-n+r)(2\eta-\omega)}{2}}\ress{\Lt{n}}{g}^{\frac{n-m+r}{2}}\ress{\Lt{n}}{\Lt{n-r}}\\
 			     &=&\beta^{\frac{n(m-n+r)(2\eta-\omega)}{2}}((-1)^{\eta n\omega}\rho^n)^{\frac{m-n+r}{2}}\ress{\Lt{n}}{\Lt{n-r}}.
 \end{eqnarray*}
Thus,
\begin{equation}\label{Ecuacion:auxiliar}
\ress{\Lt{n}}{\Lt{m}} =\left((-1)^{\eta \omega}\beta^{2\eta-\omega}\rho\right)^{\frac{(m-n+r)}{2}n}\ress{\Lt{n-r}}{\Lt{n}}.
\end{equation}
Since $n(n-r)$ is even, we have that 	$\ress{\Lt{n-r}}{\Lt{n}}=\ress{\Lt{n}}{\Lt{n-r}}$. This and $m>n-r$ imply that  $\ress{\Lt{n}}{\Lt{n-r}}$ satisfies \eqref{resultLnm1}.
From this and \eqref{Ecuacion:auxiliar} we have	
 \begin{eqnarray*}
 	\ress{\Lt{n}}{\Lt{m}}&=&\left((-1)^{\eta \omega}\beta^{2\eta-\omega}\rho\right)^{\frac{(m-n+r)}{2}n}2^{\eta \gcd(n-r,n)}\left((-1)^{\eta\omega}\beta^{2\eta-\omega}\rho\right)^{\frac{n(n-r)}{2}}\\
 	&=&2^{\eta \gcd(n,m)}\left((-1)^{\eta\omega}\beta^{2\eta-\omega}\rho\right)^{\frac{nm}{2}}.
 \end{eqnarray*}
 Thus, $\ress{\Lt{n}}{\Lt{m}}$ satisfies \eqref{resultLnm1}. That is a contradiction.

\textbf{Case $q$ is even} Let  $q=2t$. Note that $t= \lceil q/2\rceil$. Using Proposition \ref{Dic2} Part for the even case,
Lemma \ref{Properties:res} Parts  \eqref{Rbasic-2} and \eqref{Rbasic-3} and following a similarly procedure as in the
proof of the case $q=2t+1$ we obtain 
$\ress{\Lt{n}}{\Lt{m}}=\beta^{(2\eta-\omega)(m-r)n/2}\ress{\Lt{n}}{(-1)^{(n+1)t}}\ress{\Lt{n}}{g^{nt}\Lt{r}}$.
This, the fact that  $\ress{\Lt{n}}{(-1)^{(n+1)t}}=1$ and following a similar procedure as in the proof of the case $q=2t+1$ we obtain that
$\ress{\Lt{n}}{\Lt{m}}=((-1)^{\eta\omega}\beta^{2\eta-\omega}\rho)^{(m-r)n/2}\ress{\Lt{r}}{\Lt{n}}$.
Since $r<n$, we have $r\notin B$. Therefore, $r\in A$. This implies that
\begin{eqnarray*}
	\ress{\Lt{n}}{\Lt{m}}&=& ((-1)^{\eta\omega}\beta^{2\eta-\omega}\rho)^{\frac{(m-r)n}{2}}\ress{\Lt{r}}{\Lt{n}}\\
    &=&((-1)^{\eta\omega}\beta^{2\eta-\omega}\rho)^{\frac{(m-r)n}{2}}\left((-1)^{\eta\omega}2^{\eta \gcd(r,n)}\beta^{2\eta-\omega}\rho\right)^{\frac{nr}{2}}\\
	&=&2^{\eta \gcd(n,m)}\left((-1)^{\eta\omega}\beta^{2\eta-\omega}\rho\right)^{\frac{mn}{2}}.
\end{eqnarray*}
Thus, $\ress{\Lt{n}}{\Lt{m}}$ satisfies \eqref{resultLnm1}. That is a contradiction. This completes the proof that $A=\mathbb{Z}_{> 0}$.

We now prove the case $\Lt{0}=1$. It is easy to see that $\overline{\Lt{n}}=2\Lt{n}$ is a GFP sequence of  Lucas-type where
	$\overline{\Lt{0}}=2$. This and the previous case imply
	\[ \ress{\overline{\Lt{m}}}{\overline{\Lt{n}}}= 
		\begin{cases} 0 & \mbox{if }E_{2}(n)=E_{2}(m);\\
		2^{\eta \gcd(m,n)}\left((-1)^{\eta\omega}\beta^{2\eta-\omega}\rho\right)^{\frac{nm}{2}} & \mbox{if }E_{2}(n)\not=E_{2}(m).
		\end{cases}
	\]
		
		Since $\ress{\overline{\Lt{m}}}{\overline{\Lt{n}}}=\ress{2\Lt{m}}{2\Lt{n}}$, we have
		\[
\ress{\overline{\Lt{m}}}{\overline{\Lt{n}}} =\ress{2}{\Lt{n}}\ress{\Lt{m}}{2}\ress{\Lt{m}}{\Lt{n}}
							     =2^{(n+m)\eta}\ress{\Lt{m}}{\Lt{n}}.
		\]
		Therefore,
		$
		\ress{\Lt{m}}{\Lt{n}}=2^{-(n+m)\eta}\ress{\overline{\Lt{m}}}{\overline{\Lt{n}}}
		$, completing the proof.
	\end{proof}	
	
\subsection{Proof of Theorem \ref{Main:3:thm}}\label{Proof:Resultants:LFT}
We prove the third main result. The proof of Proposition \ref{cor3} is similar to the proof of 
Proposition \ref{Cor:2} (this uses  Proposition \ref{second:main:thm} Part \eqref{gcd:LnFm} instead of 
Part \eqref{gcd:two:lucas}) so it is omitted.

\begin{proposition}\label{cor3}
Let $m, n\in\mathbb{Z}_{> 0}$.
$E_2(n) < E_{2}(m)$  if and only if   $\ress{\Lt{n}}{\Ft{m}}= 0$.
\end{proposition}

\begin{proof}[Proof of Theorem \ref{Main:3:thm}]   We consider two cases: $\alpha=1$ and $\alpha=2$. We prove the case $\alpha=1$,
the case $\alpha = 2$ is similar and it is omitted.
Let $$A=\{i \in \mathbb{Z}_{>0}\mid \forall j \in \mathbb{Z}_{>0}, \text{ Property \eqref{resultLnFnm1} holds for } \ress{\Lt{j}}{\Ft{i}} \}.$$
\begin{equation}\label{resultLnFnm1}
			 \ress{\Lt{j}}{\Ft{i}}=
			 	\begin{cases} 0 &  \mbox{ if }E_{2}(j)<E_{2}(i);\\
				2^{\eta \gcd(i,j)-\eta}\left((-1)^{\eta\omega}\beta^{2\eta-\omega}\rho\right)^{(j(i-1))/2}  &\mbox{if }E_{2}(j) \ge E_{2}(i).
		\end{cases}
\end{equation}	
Since $\ress{\Lt{j}}{1}=1$, we have $i=1 \in A$.  So, $A\ne \emptyset$. The following claim completes the proof.

\noindent {\bf Claim.}  $A=\mathbb{Z}_{> 0}$.

\noindent \emph{Proof of Claim.} Suppose $B:=\mathbb{Z}_{> 0} \setminus A$ is a non-empty set.  Let $m \ne 1$ be the least element of $B$.
So, there is $h\in \mathbb{Z}_{> 0}$ such that $\ress{\Lt{h}}{\Ft{m}}$ does not satisfy \eqref{resultLnFnm1}. Let $n$ be the least element of the
non-empty set   $H=\{h \in \mathbb{Z}_{>0}\mid \ress{\Lt{h}}{\Ft{m}} \text{ does not satisfy  } \eqref{resultLnFnm1}\}$.

If  $E_{2}(n) < E_{2}(m)$, then by  Proposition \ref{cor3} it holds $\ress{\Lt{n}}{\Ft{m}}= 0$. This and \eqref{resultLnFnm1} imply that $m\in A$.
That is a contradiction. Let us suppose that  $E_{2}(n) \ge E_{2}(m)$.

We now analyze cases on $m$.

{\bf Case $m=n$.} Note that $\ress{\Lt{n}}{\Ft{n}}=\ress{\Ft{n}}{\Lt{n}}$. From Proposition \ref{Fibonacci:Lucas:Identities} Part \eqref{Fibonacci:Identities1}
with $r=q=\alpha=1$  (if $\alpha\ne 1$ is similar) we have $\Lt{n}= \Ft{n+1}+g\Ft{n-1}=d\Ft{n}+2g\Ft{n-1}$.  Therefore, $\ress{\Ft{n}}{\Lt{n}}$ equals
\begin{eqnarray*}
	\ress{\Ft{n}}{d\Ft{n}+2g\Ft{n-1}} &=&(\beta^{n-1})^{\eta n-(\omega+\eta(n-2))} \ress{\Ft{n}}{2g\Ft{n-1}}\\
	&=&\beta^{(n-1)(2-\omega)}2^{\eta(n-1)} \ress{\Ft{n}}{g\Ft{n-1}}.
\end{eqnarray*}	
By Lemma \ref{Properties:res} Part \eqref{gfactor:FIbo:Fibo} and Lemma \ref{lemma:Resultant:FibonaciT} Part \eqref{resnn-1}, we have
    \begin{eqnarray*}
	\ress{\Ft{n}}{d\Ft{n}+2g\Ft{n-1}}&=&2^{\eta(n-1)}(-1)^{\omega\eta (n-1)}(\beta^{(n-1)(2\eta-\omega)})\rho^{n-1} \big((-1)^{\omega\eta}\beta^{2\eta-\omega}\rho\big)^{\frac{(n-2)(n-1)}{2}}.\\
	&=&2^{\eta(n-1)} \big((-1)^{\omega\eta}\beta^{2\eta-\omega}\rho\big)^{\frac{n(n-1)}{2}}.
\end{eqnarray*}
So, $\ress{\Lt{n}}{\Ft{n}}$ satisfies  Property \eqref{resultLnFnm1}. That is contradiction. Therefore $m\ne n$.

{\bf Case $m>n$.}  From the Euclidean algorithm we know that $m=nq+r$ for $ 0\le r<n$. We consider two sub-cases on $q$.

{\bf Sub-case $q=1$.} Note that $ 0< r<n$. So, $m=n+r$ and $\ress{\Lt{n}}{\Ft{m}}=\ress{\Lt{n}}{\Ft{n+r}}$. This and Proposition \ref{Fibonacci:Lucas:Identities} Part \eqref{Fibonacci:Identities1} imply that $\ress{\Lt{n}}{\Ft{m}}= \ress{\Lt{n}}{\alpha \Lt{n}\Ft{r}+(-g)^r \Ft{n-r}}$. Therefore,
\begin{eqnarray*}
\ress{\Lt{n}}{\Ft{m}}
                    &=&\beta^{n \eta (n+r-1)-(\omega r+\eta(n-r-1)}\ress{\Lt{n}}{(-g)^{r}\Ft{n-r}}\\
                    &=& \beta^{n(2\eta-\omega)r}\ress{\Lt{n}}{(-g)^{r}\Ft{n-r}}\\
                    &=&\beta^{n(2\eta-\omega)r}\ress{\Lt{n}}{(-g)^{r}}\ress{\Lt{n}}{\Ft{n-r}}.
 \end{eqnarray*}                  
This and Lemma \ref{Properties:res} Part \eqref{gfactor:Lucas:Lucas} imply                
\begin{eqnarray*}
\ress{\Lt{n}}{\Ft{m}}                                     
                    &=&\beta^{n(2\eta-\omega)r}\left(\ress{\Lt{n}}{-g}\right)^r\ress{\Lt{n}}{\Ft{n-r}}\\
                    &=&\beta^{n(2\eta-\omega)r}\left(\ress{\Lt{n}}{-1}\right)^r\left(\ress{\Lt{n}}{g}\right)^r\ress{\Lt{n}}{\Ft{n-r}}\\
                    &=&\beta^{n(2\eta-\omega)r}(-1)^{r\eta n}(-1)^{\eta n\omega r}\ress{g}{\Lt{n}}^{r}\ress{\Lt{n}}{\Ft{n-r}}.
\end{eqnarray*}
Therefore,
 \begin{equation} \label{Aux:Res:Ln:Fn-r}
 \ress{\Lt{n}}{\Ft{m}}=(-1)^{\eta nr(\omega+1)}\beta^{n(2\eta-\omega)r}  \rho^{nr}\ress{\Lt{n}}{\Ft{n-r}}.
 \end{equation}
 Since  $m>n-r$ is the least element of $B$, we have $n-r\in A$. Therefore it holds
 $\ress{\Lt{n}}{\Ft{n-r}}=2^{\eta(\gcd(n,n-r)-1)}\left((-1)^{\eta\omega}\beta^{2\eta-w}\rho\right)^{n(n-r-1)/2}$. This and \eqref{Aux:Res:Ln:Fn-r}
 (after simplifications) imply that
 $$\ress{\Lt{n}}{\Ft{m}}=2^{\eta(\gcd(n,n-r)-1)}\left[\beta^{2\eta-\omega}(-1)^{\eta\omega}\rho\right]^{\frac{n(n+r-1)}{2}}.$$
Therefore, $\ress{\Lt{n}}{\Ft{m}}$ satisfies Property  \eqref{resultLnFnm1}. That is a contradiction. Thus, $m\ne n+r$.

{\bf Sub-case $q>1$.}  Since $\ress{\Lt{n}}{\Ft{m}}=\ress{\Lt{n}}{\Ft{nq+r}}$, by Proposition \ref{Fibonacci:Lucas:Identities}
Part \eqref{Fibonacci:Identities1} we have $\ress{\Lt{n}}{\Ft{m}}= \ress{\Lt{n}}{\alpha \Lt{n}\Ft{n(q-1)+r}-(-g)^n \Ft{n(q-2)+r}}$. 
This and   the fact that ${n\eta (nq+r-1)-(\omega n+\eta(n(q-2)+r-1))}= (2\eta-\omega)n^2$, imply that  
 $$\ress{\Lt{n}}{\Ft{m}} = \beta^{n(2\eta-\omega)n}\ress{\Lt{n}}{-(-g)^{n}\Ft{n(q-2)+r}}.$$
So,
\begin{eqnarray*}
\ress{\Lt{n}}{\Ft{m}} &=&\beta^{(2\eta-\omega)n^2}\ress{\Lt{n}}{-1}\ress{\Lt{n}}{(-g)^{n}}\ress{\Lt{n}}{\Ft{n(q-2)+r}}\\
                    	      &=&\beta^{(2\eta-\omega)n^2}(-1)^{\eta n}\left(\ress{\Lt{n}}{-g}\right)^n\ress{\Lt{n}}{\Ft{n(q-2)+r}}\\
                    	      &=&\beta^{(2\eta-\omega)n^2}(-1)^{\eta n}(-1)^{\eta n^2}\left(\ress{\Lt{n}}{g}\right)^n\ress{\Lt{n}}{\Ft{n(q-2)+r}}\\
	                       &=&\beta^{(2\eta-\omega)n^2}(-1)^{\eta n(n+1)}\left(\ress{\Lt{n}}{g}\right)^n\ress{\Lt{n}}{\Ft{n(q-2)+r}} \\
	                       &=&\beta^{(2\eta-\omega)n^2}\left(\ress{\Lt{n}}{g}\right)^n\ress{\Lt{n}}{\Ft{n(q-2)+r}}.
\end{eqnarray*}
Since $m$ is the least element of $B$ and $n(q-2)+r<m$, we have that $n(q-2)+r\in A$. This and  $E_{2}(n) \ge E_{2}(m) = E_{2}(n(q-2)+r)$ imply that
$$\ress{\Lt{n}}{\Ft{n(q-2)+r}} =2^{\eta(\gcd(n,n(q-2)+r)-1)}\left[(-1)^{\eta\omega}\beta^{2\eta-\omega}\rho \right]^\frac{n(n(q-2)+r-1)}{2}.$$
Note that $\gcd(n,n(q-2)+r)=\gcd(n,nq+r)$. Therefore,
\begin{eqnarray*}
\ress{\Lt{n}}{\Ft{m}}&=&\beta^{(2\eta-\omega)n^2}(-1)^{\eta n\omega}\rho^{n^2}2^{\eta(\gcd(n,n(q-2)+r)-1)}\left[(-1)^{\eta\omega}\beta^{2\eta-\omega}\rho \right]^\frac{n(n(q-2)+r-1)}{2}\\
                    &=&\beta^{\frac{(2\eta-\omega)n(nq+r-1)}{2}}(-1)^{\frac{\eta\omega n(nq+r-1)}{2}}\rho^{\frac{n(nq+r-1)}{2}}2^{\eta(\gcd(n,nq+r)-1)}\\
                    &=&2^{\eta(\gcd(n,nq+r)-1)}\left[(-1)^{\eta\omega}\beta^{2\eta-\omega}\rho\right]^\frac{n(nq+r-1)}{2}.
\end{eqnarray*}
From this we conclude that $ \ress{\Lt{n}}{\Ft{m}}=2^{\eta(\gcd(n,m)-1)}\left[\beta^{2\eta-\omega}(-1)^{\eta\omega}\rho\right]^{\frac{n(m-1)}{2}}$. Therefore,
$m\in A$.  That is is a contradiction. 

{\bf Case $m<n$.} From the Euclidean algorithm we know that $n=mq+r$ for $ 0\le r<m$. We consider two sub-cases on $q$.

{\bf Sub-case $q=1$.} The case $q> 1$ is similar and it is omitted.  In this case $r\ne 0$, for $r=0$ see the case $m=n$. 
Therefore, $\ress{\Lt{n}}{\Ft{m}}=\ress{\Lt{m+r}}{\Ft{m}}$. This, Proposition
\ref{Fibonacci:Lucas:Identities}  Part \eqref{Lucas:Identities2} and Lemma \ref {Properties:res} Parts \eqref{Rbasic-1} imply that
\begin{eqnarray*}
\ress{\Lt{n}}{\Ft{m}}&=&\ress{\left((a-b)^{2}/{\alpha}\right)\Ft{m}\Ft{r}+(-g)^{r}\Lt{m-r}}{\Ft{m}}\\
			     & =&(-1)^{\eta^2(m+r)(m-1)}\ress{\Ft{m}}{\frac{(a-b)^2}{\alpha}\Ft{m}\Ft{r}+(-g)^r\Lt{m-r}}.
\end{eqnarray*}

Note that $(m\pm r)(m-1)$ and $r(m-1)$ are even (it is clear if $m$ is odd), if $m$ is even, then $1\le E_2(m)\le  E_2(n=m+r)$.
So, both  $n$ and $r$ are even. Therefore, $(-1)^{\eta^2(m+r)(m-1)}=(-1)^{\eta(m-1)r}=(-1)^{\eta(m-1)\omega r}=1$.

From Lemma \ref {Properties:res} Parts \eqref{Rbasic-2}, \eqref{Rbasic-2add}, and  \eqref{Rbasic-3} we have
\begin{eqnarray*}
\ress{\Lt{n}}{\Ft{m}}
&=&(\beta^{m-1})^{\eta(m+r)-(\eta(m-r)+\omega r)}\ress{\Ft{m}}{(-g)^r\Lt{m-r}}\\
&=&(\beta^{m-1})^{(2\eta -\omega)r}\ress{\Ft{m}}{(-g)^r}\ress{\Ft{m}}{\Lt{m-r}}\\
&=&(\beta^{m-1})^{(2\eta -\omega)r}\ress{\Ft{m}}{(-g)^r}\ress{\Ft{m}}{\Lt{m-r}}\\
&=&(\beta^{m-1})^{(2\eta -\omega)r}\ress{\Ft{m}}{(-1)^r}\ress{\Ft{m}}{g^r}\ress{\Ft{m}}{\Lt{m-r}}\\
&=&(\beta^{m-1})^{(2\eta -\omega)r}(-1)^{\eta(m-1)r}(-1)^{\eta(m-1)\omega r}\ress{g}{\Ft{m}}^{r}\ress{\Ft{m}}{\Lt{m-r}}\\
&=&(\beta^{m-1})^{(2\eta -\omega)r} \rho^{r(m-1)}(-1)^{\eta^2(m-1)(m-r)}\ress{\Lt{m-r}}{\Ft{m}}\\
&=&(-1)^{\eta(m-1)(1+\omega)r} (\beta^{m-1})^{(2\eta -\omega)r}\rho^{r(m-1)}\ress{\Lt{m-r}}{\Ft{m}}.
\end{eqnarray*}
Since $n=m+r$ is the least element of $H$, we have that $(m-r)\not\in H$.
Therefore, $\ress{\Lt{m-r}}{\Ft{m}}$ satisfies  \eqref{resultLnFnm1}. So, 
$$\ress{\Lt{n}}{\Ft{m}}= \beta^{(m-1)(2\eta -\omega)r}\rho^{r(m-1)} 2^{\eta(\gcd(m-r,m)-1)}\left(\beta^{2\eta-\omega}\rho\right)^{(m-r)(m-1)/2}.$$ Since $0<r<m$, the
$\gcd(m-r,m)=1$. This (after some simplifications) implies that
$\ress{\Lt{n}}{\Ft{m}}=\left((-1)^{\eta\omega}\beta^{2\eta-\omega}\rho\right)^{(m-1)(m+r)/2}$. Therefore,
$m\in A$.  That is a contradiction. This completes the proof of the claim.
\end{proof}

\section{Derivatives of GFP}

In this section we  give closed formulas for the derivatives of GFPs. The derivatives of a Lucas-type polynomials is given in term of its equivalent polynomial 
and the derivative of a Fibonacci-type polynomial is given in terms of Fibonacci-type and its equivalent. 
The derivative of the familiar polynomials studied here are in Table \ref{Deriv:Gen:Falcon:Plaza}.

Theorem \ref{Gen:Falcon:Plaza} is a generalization of the derivative given by several authors 
\cite{Richard, Falcon, Horadam1, Horadam2,Horadam3,wang}  for some Fibonacci-type polynomials and some Lucas-type polynomials.
Recall that from \eqref{bineformulauno} and \eqref{bineformulados} we have $d=a+b$, $b= -g/a$ where $d$ and $g$ are the polynomials defined in
\eqref{Fibonacci;general:FT} and \eqref{Fibonacci;general:LT}. This implies that $a-b= a+ga^{-1}$. Here we use
$\Ft{n}^{\prime}$, $\Lt{n}^{\prime}$, $a^{\prime}$, $b^{\prime}$ and $d^{\prime}$ to mean the derivatives of
$\Ft{n}$, $\Lt{n}$, $a$, $b$ and $d$ with respect to $x$.

Evaluating the derivative of Fibonacci polynomials and the derivative of Lucas polynomials at $x=1$ and $x=2$ we obtain numerical sequences that appear in Sloan \cite{sloane}. Thus, 
\begin{eqnarray*}
\frac{d(F_{n})}{dx}\Bigr|_{\substack{x=1}} =\seqnum{A001629};\quad&& \quad  \frac{d(F_{n})}{dx}\Bigr|_{\substack{x=2}}=\seqnum{A006645};\\
\frac{d(D_{n})}{dx}\Bigr|_{\substack{x=1}} = \seqnum{A045925}; \quad &&\quad  \frac{d(D_{n})}{dx}\Bigr|_{\substack{x=2}} = \seqnum{A093967}.
\end{eqnarray*}

For the sequences generated by the derivatives of the other familiar polynomials studied here see:  \seqnum{A001871}, \seqnum{A317403}, \seqnum{A317404}, \seqnum{A317405}, \seqnum{A317408},  \seqnum{A317449}, \seqnum{A317450}, and \seqnum{A317451}.

\begin{theorem}\label{Gen:Falcon:Plaza} If $g$ is a constant, then
\begin{enumerate}[(i)]
  \item \label{Deriv:Fibo} $$\Ft{n}^{\prime}=\frac{d^{\prime} \left(n g \Ft{n-1}-d\cdot \Ft{n}+n \Ft{n+1}\right)}{(a-b)^2}=\frac{d^{\prime} \left(n \alpha \Lt{n}-d\Ft{n}\right)}{(a-b)^2}.$$
  \item \label{Deriv:Lucas}$$\Lt{n}^{\prime}= \frac{ n d^{\prime}  \Ft{n}}{\alpha}.$$
\end{enumerate}
\end{theorem}

\begin{proof} We prove Part \eqref{Deriv:Fibo}.
From Binet formula \eqref{bineformulauno} and $b= -g/a$ we have $$\Ft{n}=\left(a^{n}-(-g)^na^{-n}\right)/\left(a-(-g)a^{-1}\right).$$
Differentiating $\Ft{n}$ with respect to $x$, using= $a-b= a+ga^{-1}$ and simplifying we have
\begin{equation}\label{eq:derivative}
\Ft{n}^{\prime}=\dfrac{n a^{\prime}\left(a^{n-1}+(-g)^{n}a^{-n-1} \right)}{(a-b)^2}-\dfrac{a^{\prime}(1-g a^{-2}) \left(a^{n}-(-g)^{n}a^{-n} \right)}{(a-b)^2}.
\end{equation}
Since $d=a+b$, and $b= -g/a$, we have $a^{\prime}+b^{\prime}=d^{\prime}$, and $b^{\prime}= g a^{-2} a^{\prime}$. These imply that
$$a^{\prime}=\frac{a d^{\prime}}{a+ga^{-1}} \quad \text{ and }\quad 1-g a^{-2}=\frac{d}{a}.$$
Substituting these results in \eqref{eq:derivative} and simplifying  we have
$$\Ft{n}^{\prime}=\dfrac{n a d^{\prime} \left(a^{n-1}+(-g)^{n}a^{-n-1} \right)}{(a-b)^2}-\dfrac{d \cdot d^{\prime}}{(1+g a^{-1})^2}
\frac{ \left(a^{n}-(-g)^{n}a^{-n} \right)}{(a-b)}.$$
Thus,
$$\Ft{n}^{\prime}=\dfrac{n d^{\prime} \left(a^{n}+b^{n} \right)}{(a-b)^2}-\dfrac{d \cdot d^{\prime}}{(a-b)^2}
\frac{ \left(a^{n}-(-g)^{n}a^{-n} \right)}{(a-b)}.$$
It is known that (see for example \cite{florezMcanallyMuk}) $a^n+b^n=g\Ft{n-1}+\Ft{n+1}$. So,
$$\Ft{n}^{\prime}=\dfrac{n d^{\prime} \left(g\Ft{n-1}+\Ft{n+1} \right)-d \cdot d^{\prime}\Ft{n} }{(a-b)^2}.$$
This completes the proof of Part \eqref{Deriv:Fibo}.

We now prove Part \eqref{Deriv:Lucas}. From \cite{florezMcanallyMuk} we know that $\Lt{n}= \left( g \Ft{n-1}+\Ft{n+1}\right)/\alpha$.
Differentiating $\Lt{n}$ with respect to $x$, we have (recall that $g$ is constant)
$\Lt{n}^{\prime}= \left( g \Ft{n-1}^{\prime}+\Ft{n+1}^{\prime}\right)/\alpha$. This and Part \eqref{Deriv:Fibo} imply that
$$\Lt{n}^{\prime}= \frac{g d^{\prime} \left((n-1) \alpha \Lt{n-1}-d\Ft{n-1}\right)}{\alpha (a-b)^2} +\frac{d^{\prime} \left((n+1) \alpha \Lt{n+1}-d\Ft{n+1}\right)}{\alpha (a-b)^2}. $$
Simplifying we have
$$\Lt{n}^{\prime}= \frac{d^{\prime}}{\alpha (a-b)^2} \left( (n-1) \alpha g \Lt{n-1} + (n+1) \alpha \Lt{n+1}-d \alpha \frac{g\Ft{n-1} +\Ft{n+1}}{\alpha}\right). $$
This and $\Lt{n}= \left( g \Ft{n-1}+\Ft{n+1}\right)/\alpha$ imply that
$$\Lt{n}^{\prime}= \frac{d^{\prime} \left( (n-1) g \Lt{n-1} + (n+1) \Lt{n+1}-d \Lt{n} \right)}{ (a-b)^2} . $$
Therefore,
\begin{equation}\label{Lucas:derivative:short}
\Lt{n}^{\prime}= \frac{d^{\prime} \left( n (g \Lt{n-1} + \Lt{n+1})+( \Lt{n+1}-g \Lt{n-1} +)-d \Lt{n} \right)}{ (a-b)^2} .
\end{equation}
From \cite{florezMcanallyMuk} we know that $$ g \Lt{n-1} + \Lt{n+1}= (a-b)^2 \Ft{n}/\alpha, \quad  \Lt{n+1}-g \Lt{n-1}= \alpha \Lt{n}\Lt{1}, \quad\text{ and } \quad \alpha \Lt{1}-d=0.$$
Substituting these identities in   \eqref{Lucas:derivative:short} completes the proof.
\end{proof}

\begin{table} [!ht]
\begin{center}\scalebox{1}{
\begin{tabular}{|l|l||l|l|} \hline
 			Fibonacci-type	 	 &  Derivative           	&Lucas-Type   		&  Derivative           \\ \hline \hline
			Fibonacci            	 & $\frac{d(F_{n})}{dx}=\frac{n D_{n}-xF_{n}}{4+x^2}$  & Lucas               	& $\frac{d(D_{n})}{dx}=nF_{n}$         \\[3pt]
 			Pell		         	 & $\frac{d(P_{n})}{dx}=\frac{n Q_{n}-2xP_{n}}{2(1+x^2)}$	& Pell-Lucas-prime   & $ \frac{d(Q_{n})}{dx}=2nP_{n}$       \\[3pt]
 			Fermat		         & $\frac{d(\Phi_{n})}{dx}=\frac{3(n\vartheta_{n}-3x \Phi_{n})}{-8 + 9 x^2}$ & Fermat-Lucas        & $\frac{ d(\vartheta_{n})}{dx}=3n\Phi_{m}$  \\[3pt]
 			Chebyshev 2nd kind  & $\frac{d(U_{n})}{dx}=\frac{2 n T_{n} -2x U_{n}}{2 \left(x^2-1\right)}$   & Chebyshev 1st kind  & $\frac{ d (T_{n})}{dx}=n U_{m}$\\[3pt]
 			Morgan-Voyce           &$  \frac{d(B_{n})}{dx}=\frac{n C_{n}-(x+2)B_{n}}{x (x+4)}$   & Morgan-Voyce        &$\frac{ d (C_{n})}{dx}=n B_{m}$    \\[3pt] \hline 
\end{tabular}}
\end{center}
\caption{Derivatives of GFP using Theorem \ref{Gen:Falcon:Plaza}.} \label{Deriv:Gen:Falcon:Plaza}
\end{table}

\section{Proofs of main results about the discriminant}

Recall that one of the expressions for the discriminant of a polynomial $f$ is given by
$\dis(f)=(-1)^{n(n-1)/2}a^{-1}\ress{f}{f^{\prime}}$ where $a=\lc(f)$, $n=\deg(f)$ and $f^{\prime}$ the derivative of $f$.

\begin{lemma}\label{lemma1:Disc:Mod:d24g}
	For  $n\in \mathbb{Z}_{\ge 0}$ this holds
	$$
	\Ft{n}\bmod{d^2 +4g}\equiv  \begin{cases}n(-g)^{(n-1)/2} & \mbox{if $n$  is odd, }\\
					   (-1)^{(n+2)/2}\left(ndg^{(n-2)/2}\right)/2 & \mbox{if  $n$  is even.}
	\end{cases}
	$$
\end{lemma}

\begin{proof} We use mathematical induction.  Let $S(k)$ be statement:
$$
	\Ft{k}\bmod{d^2 +4g}\equiv 
	\begin{cases}(-1)^{(k-1)/2} k g^{(k-1)/2} & \mbox{if $k$  is odd, }\\
										 (-1)^{(k+2)/2} \left(kdg^{(k-2)/2}\right)/2 & \mbox{if  $k$  is even.}
	\end{cases}
$$
Since $\Ft{1}=1$ and $\Ft{2}=d$, we have $S(1)$ and $S(2)$ are true. Suppose that the statement is true for some $k=n-1$ and $k=n$.
Thus, suppose that $S(n-1)$ and  $S(n)$ are true and we prove $S(n+1)$.
We consider two cases on the parity of $n$.

{\bf Case $n$ even.} Recall that  $ \Ft{n+1}=d\Ft{n}+g\Ft{n-1}$. This and $S(n-1)$ and  $S(n)$ (with $n$ even and $n-1$ odd) imply that
$\Ft{n+1}\equiv (-1)^{(n+2)/2} \left(nd^2g^{(n-2)/2}/2\right)+ (n-1)(-g)^{(n-2)/2}g\bmod{d^2+4g}$. Simplifying
\[
 \Ft{n+1}\equiv (-1)^{(n+2)/2}\frac{nd^2g^{(n-2)/2}}{2}+ (2n-(n+1))(-1)^{(n-2)/2}g^{n/2}\bmod{d^2+4g}.
\]
It is easy to see that
\[
(-1)^{(n+2)/2}\frac{nd^2g^{(n-2)/2}}{2}+ 2n(-1)^{(n-2)/2}g^{n/2}= (-1)^{(n+2)/2}\frac{ng^{(n-2)/2}}{2}\left( d^2+4g\right).
\]
Thus,
\[ 
\Ft{n+1} \equiv (-1)^{(n+2)/2}\frac{ng^{(n-2)/2}}{2}\left( d^2+4g\right)+(n+1)(-g)^{n/2} \bmod {d^2+4g}.
\]
This implies that  $\Ft{n+1}\equiv (n+1)(-g)^{n/2}\bmod {d^2+4g}$.

{\bf Case $n$ odd.}  $S(n-1)$ and  $S(n)$ (with $n$ odd and $n-1$ even) and $ \Ft{n+1}=d\Ft{n}+g\Ft{n-1}$, imply that
\begin{eqnarray*}
 \Ft{n+1}	&\equiv &  n(-g)^{(n-1)/2}d+(-1)^{(n+1)/2}\left(\frac{(n-1)dg^{(n-3)/2}}{2}\right)g  \bmod{d^2+4g}\\
 		&\equiv & d g^{(n-1)/2}\left(\frac{(-1)^{(n-1)/2}2n-(-1)^{(n-1)/2}(n-1)}{2} \right) \bmod{d^2+4g}\\
		&\equiv & \frac{(-1)^{(n+3)/2}(n+1)d g^{(n-1)/2}}{2} \mod{d^2+4g}.\end{eqnarray*}
This completes the proof.
\end{proof}

\begin{lemma} \label{Disc:Lemma18}
If $n\in \mathbb{Z}_{\ge 0}$, then
$
\ress{(a-b)^2}{\Ft{n}}=(\beta^{2\eta-\omega}\rho)^{(n-1)}n^{2\eta}.
$
\end{lemma}
\begin{proof} From \cite{florezMcanallyMuk} we know that  $(a-b)^2=d^2+4g$. This and Lemma \ref{lemma1:Disc:Mod:d24g} imply that there is a
polynomial $T$ such that
\begin{equation}\label{Formula:Lemma18}
\Ft{n}= \begin{cases}(a-b)^2T+n(-g)^{(n-1)/2} & \mbox{if $n$ is odd, }\\
(a-b)^2T+(-1)^{(n+2)/2}2^{-1}dg^{(n-2)/2}n & \mbox{if $n$ is even.}
	\end{cases}
\end{equation}
Using Lemma \ref{Properties:res} Parts \eqref{Rbasic-1}, \eqref{Rbasic-2add} and \eqref{Rbasic-3} and simplifying we have
\begin{equation}\label{ident20}
\ress{d^2+4g}{g^m}= \ress{d^2+4g}{g}^m =\ress{d^2+4g}{g}^m =(\lambda^{2\eta-2\eta}\ress{g}{d}^2)^m =\rho^{2m}.
\end{equation}
To find $\ress{(a-b)^2}{\Ft{n}}$ we consider two cases, depending on the parity of $n$.

{\bf Case $n$ is even.}  From \eqref{Formula:Lemma18} we have
$$\ress{(a-b)^2}{\Ft{n}}=\ress{(a-b)^2}{(a-b)^2T+(-1)^{(n+2)/2}2^{-1}dg^{(n-2)/2}n }.$$
This and Lemma \ref {Properties:res} Parts \eqref{Rbasic-1}, \eqref{Rbasic-2} and \eqref{Rbasic-3} imply that
\begin{eqnarray*}
\ress{(a-b)^2}{\Ft{n}} &=&\beta^{(2\eta-\omega)(n-2)}\ress{(a-b)^2}{(-1)^{(n+2)/2}2^{-1}n}\ress{(a-b)^2}{dg^{(n-2)/2}}\\
                      &=&\beta^{(2\eta-\omega)(n-2)}(2^{-1}n)^{2\eta}\ress{(a-b)^2}{d}\ress{(a-b)^2}{g^{(n-2)/2}}\\
                      &=&\beta^{(2\eta-\omega)(n-2)}(2^{-1}n)^{2\eta}\ress{d}{d^2+4g}\ress{(a-b)^2}{g^{(n-2)/2}}.
\end{eqnarray*}
 Using similar analysis as in  \eqref{ident20} we have $\ress{(a-b)^2}{g^{(n-2)/2}}=\rho^{n-2}$.
 It  is easy to see that $\ress{(a-b)^2}{d}=\ress{d}{d^2+4g}=\beta^{2n-\omega}2^{\eta}$. Therefore,
$$
\ress{(a-b)^2}{\Ft{n}} =\beta^{(2\eta-\omega)(n-1)}n^{2\eta}\rho\rho^{n-2} =(\beta^{2\eta-\omega}\rho)^{(n-1)}n^{2\eta}.
$$

{\bf Case $n$ is odd.} From \eqref{Formula:Lemma18} we have
$$\ress{(a-b)^2}{\Ft{n}}=\ress{(a-b)^2}{(a-b)^2T+(-1)^{(n+2)/2}2^{-1}dg^{(n-2)/2}n }.$$
This and Lemma \ref {Properties:res} Parts \eqref{Rbasic-1}, \eqref{Rbasic-2} and \eqref{Rbasic-3} imply that
\begin{eqnarray*}
\ress{(a-b)^2}{\Ft{n}}
&=&(\beta^2)^{\eta(n-1)-\omega(n-1)/2}\ress{d^2+4g}{n(-g)^{(n-1)/2}}\\
&=&\beta^{(2\eta-\omega)(n-1)}n^{2\eta}\ress{d^2+4g}{g^{(n-1)/2}}\\
&=&(\beta^{2\eta-\omega}\rho)^{(n-1)}n^{2\eta}.
\end{eqnarray*}
This completes the proof.
\end{proof}

\subsection{Proof of Theorems \ref{Disc:Fibonacci:type} and \ref{Disc:Lucas:type}}\label{Proof:Disc:FT}
We now prove the last two main results.

\begin{proof}[Proof of Theorem \ref{Disc:Fibonacci:type}] From Theorem \ref{Gen:Falcon:Plaza} we have $$\ress{\Ft{n}}{(a-b)^2\Ft{n}^{\prime}}=\ress{\Ft{n}}{d^{\prime}(n\alpha \Lt{n}-d\Ft{n})}.$$
Since $\deg(d)=1$, we have that $d^{\prime}$ is a constant. (Recall that when $\Ft{n}$ and $\Lt{n}$ are together in a resultant, they are equivalent.) Therefore,
$\ress{\Ft{n}}{(a-b)^2\Ft{n}^{\prime}} =(d^{\prime})^{n-1}\ress{\Ft{n}}{n\alpha \Lt{n}-d\Ft{n}}$.  Since $\deg(\Ft{n})=\eta(n-1)$ and $\deg(\Lt{n})=\eta(n)$, we have
$\ress{\Ft{n}}{n\alpha \Lt{n}-d\Ft{n}}=\ress{\Ft{n}}{n\alpha\Lt{n}}$. So, $\ress{\Ft{n}}{(a-b)^2\Ft{n}^{\prime}}=(\alpha  d^{\prime} n)^{n-1}\ress{\Ft{n}}{\Lt{n}}
=(\alpha  d^{\prime} n)^{n-1}\ress{\Lt{n}}{\Ft{n}}$. This and Theorem \ref{Main:3:thm} imply that
\begin{equation}\label{Half:formula:Disc}
\ress{\Ft{n}}{(a-b)^2\Ft{n}^{\prime}}
=(\alpha  d^{\prime} n)^{n-1} 2^{n-1} \alpha^{1-n}\big(\beta^{2}\rho\big)^{n(n-1)/2}\\
=(2d^{\prime} n)^{n-1} \big(\beta^{2}\rho\big)^{n(n-1)/2}.
\end{equation}
On the other hand, from Lemma \ref{Disc:Lemma18} and the fact that  $\deg(a-b)^2$ is even we have
$$
\ress{\Ft{n}}{(a-b)^2\Ft{n}^{\prime}}=\ress{\Ft{n}}{(a-b)^2}\ress{\Ft{n}}{\Ft{n}^{\prime}}=n^{2}(\beta^{2}\rho)^{(n-1)}\ress{\Ft{n}}{\Ft{n}^{\prime}}.
$$
This and \eqref{Half:formula:Disc} imply that
$$
\ress{\Ft{n}}{\Ft{n}^{\prime}}=\dfrac{(2d^{\prime} n)^{n-1} \big(\beta^{2}\rho\big)^{n(n-1)/2}}{n^{2}(\beta^{2}\rho)^{(n-1)}}=n^{n-3}(2d^{\prime})^{n-1}(\beta^{2}\rho)^{(n-1)(n-2)/2}.
$$
Therefore,
$$
\dis{\Ft{n}}=(-1)^{\frac{(n-1)(n-2)}{2}}\beta^{1-n}\ress{\Ft{n}}{\Ft{n}^{\prime}}=\beta^{1-n}
n^{n-3}(2d^{\prime})^{n-1}(-\beta^{2}\rho)^{(n-1)(n-2)/2}.
$$
This completes the proof.
\end{proof}

\begin{proof} [Proof of Theorem \ref{Disc:Lucas:type}]  From the definition of the discriminant we have
$$\dis{\Lt{n}}=(-1)^{n( n-1)/2}\alpha\beta^{-n}\ress{\Lt{n}}{\Lt{n} ^{\prime}}.$$ This and
Theorem \ref{Gen:Falcon:Plaza} imply that $\dis{\Lt{n}}=(-1)^{n( n-1)/2}\alpha\beta^{-n}\ress{\Lt{n}}{(nd^{\prime}\Ft{n})/{\alpha}}$.
Since $(nd^{\prime})/{\alpha}$ is a constant,  $\dis{\Lt{n}}=(-1)^{n( n-1)/2}\alpha\beta^{-n} (nd^{\prime}/{\alpha})^{ n} \ress{\Lt{n}}{\Ft{n}}$.
This and Theorem \ref{Main:3:thm} imply that
\begin{eqnarray*}
	\dis{\Lt{n}}
    &=&(-1)^{\frac{ n( n-1)}{2}}\alpha\beta^{-n}\Big(\frac{nd^{\prime}}{\alpha}\Big)^{ n}   2^{n-1} \alpha^{1-n}\big(\beta^{2}\rho\big)^{(n(n-1))/2}\\
    &=&\beta^{n(n-2)}\Big(nd^{\prime}\Big)^{ n}   2^{n-1} \alpha^{2-2n}\big(-\rho\big)^{(n(n-1))/2}.
\end{eqnarray*}
Completing the proof.
\end{proof}

{\bf Open question.} In this paper we did not investigate the case $\deg(g)\ge \deg(d)$. This property is satisfied by Jacobsthal polynomials. 

\section{Acknowledgement}

The first author was partially supported by Grant No 344524, 2018; The Citadel Foundation, Charleston SC.


\end{document}